\documentclass[11pt,a4paper]{amsart}
\usepackage[utf8]{inputenc}
\usepackage[T1]{fontenc}
\usepackage{amsmath,amssymb,amsthm,bm}
\usepackage{graphicx}
\usepackage{hyperref}
\usepackage{geometry}
\usepackage{xcolor}
\usepackage{enumerate}

\usepackage{verbatim}

\newtheorem{theorem}{Theorem}[section]
\newtheorem*{mainresult}{Theorem (sketch)}
\newtheorem{lemma}[theorem]{Lemma}
\newtheorem{prop}[theorem]{Proposition}
\newtheorem{corollary}[theorem]{Corollary}
\theoremstyle{definition}
\newtheorem{definition}[theorem]{Definition}
\newtheorem{assumption}[theorem]{Assumption}
\theoremstyle{remark}
\newtheorem{remark}[theorem]{Remark}

\newcommand{\N}{\mathbb{N}}
\newcommand{\Z}{\mathbb{Z}}

\newcommand{\R}{\mathbb{R}} 
\newcommand{\C}{\mathbb{C}} 

\newcommand{\E}{\mathbb{E}}

\renewcommand{\P}{\mathbb{P}}
\newcommand{\<}{\langle}
\renewcommand{\>}{\rangle}
\newcommand{\Span}{\operatorname{span}}
\newcommand{\st}{\,:\,}
\newcommand{\randgram}{\widehat\Sigma_\Omega}
\newcommand{\tr}{\operatorname{tr}}
\newcommand{\supp}{\operatorname{supp}}
\newcommand{\ran}{\operatorname{Ran}}

\author[L. Finotti]{Luca Finotti} 
\address{MaLGa Center, Department of Mathematics, University of Genoa, Via Dodecaneso 35, 16146, Genova, Italy}
\email{luca.finotti@edu.unige.it}

\author[M. Santacesaria]{Matteo Santacesaria}
\address{MaLGa Center, Department of Mathematics, University of Genoa, Via Dodecaneso 35, 16146, Genova, Italy}
\email{matteo.santacesaria@unige.it}

\title{Stochastic Generalized Sampling}


\begin{document}

\begin{abstract}
    Reconstructing an infinite-dimensional signal from a finite set of measurements is a fundamental problem in approximation theory and signal processing. While the generalized sampling (GS) framework provides a robust methodology for recovering elements in arbitrary separable Hilbert spaces, deterministic approaches suffer from severe basis-dependent dimensionality constraints, often requiring a quadratic sample complexity $m \gtrsim n^2$ to avoid numerical instability. In this paper, we introduce a fully stochastic framework for GS that natively overcomes these deterministic barriers. By drawing measurements according to an optimal leverage-score probability distribution, we prove that stable recovery is guaranteed with high probability at a near-linear sample complexity of $m \gtrsim n\log n$. Crucially, this optimal rate is universal—independent of the specific choice of measurement and reconstruction bases—and holds even when the sensing system is a highly redundant frame. To establish these guarantees, we derive a novel matrix Bernstein inequality for random rectangular operators, allowing us to rigorously control the aliasing error governed by the empirical cross-term. Finally, we demonstrate the practical efficacy of our approach on the classical problem of recovering analytic functions from continuous Fourier measurements via Legendre polynomials, where our randomized method achieve near-exponential convergence rates.
\end{abstract}

\maketitle

\section{Introduction}

A fundamental problem in sampling theory revolves around the reconstruction of an unknown signal from a finite set of measurements. The starting point of the modern theory is the celebrated Nyquist-Shannon sampling theorem, which allows one to perfectly reconstruct a bandlimited signal from its uniform point samples. However, this pioneering result is inherently limited in practice, as it requires an infinite number of samples on a rigid uniform grid. To overcome these practical limitations and accommodate more general sampling geometries, a variety of approaches have been developed, ranging from classical deterministic methods to modern stochastic strategies.

Deterministic approaches generally seek to recover an unknown signal by computing a least-squares approximation within a chosen finite-dimensional subspace. The measurements can be pointwise evaluations, as in Shannon's theorem, or modeled via abstract continuous linear functionals. In \cite{adcock2012generalized}, Adcock and Hansen introduced Generalized Sampling (GS), focusing on reconstructing an element $f$ in an infinite-dimensional separable Hilbert space $\mathcal H$ from its first $m$ measurements $(\<f,s_k\>)_{k=1}^m$, where $\{s_k\}_{k\in\N}$ is a given Riesz basis (or a frame, as extended in \cite{adcock2012frames}). The recovery is achieved by projecting onto the span of the first $n$ elements of a reconstruction system $\{w_k\}_{k\in\N}$, which the user can tailor to the specific application. The central goal is to establish a deterministic \textit{sampling rate}—an estimate of how large the sampling budget $m$ must be relative to the reconstruction dimension $n$ to guarantee stable recovery. 

Within the deterministic GS framework \cite{adcock2012generalized,adcock2012stable,adcock2012frames,adcock2013beyond}, the stable sampling rate is not universal; it is heavily basis-dependent. For highly compatible pairs, such as Fourier sampling and wavelet reconstruction, one can achieve an optimal linear scaling $m\gtrsim n$. However, other physically motivated choices—such as Fourier measurements and Legendre polynomial reconstruction—suffer from a severe quadratic bottleneck $m\gtrsim n^2$ \cite{adcock2015generalized}.

To break this deterministic barrier, one can transition to a stochastic scenario. The geometric setup remains identical, but rather than deterministically selecting the first $m$ samples, one defines a probability distribution $p=(p_j)_{j\in\N}$ over the natural numbers and draws $m$ independent sample indices. Instead of seeking a worst-case deterministic rate, we fix a confidence parameter $\delta\in(0,1)$ and establish a \textit{stochastic sampling rate} that guarantees stable recovery with a probability exceeding $1-\delta$. 

Our analysis falls within the broader context of stochastic weighted least-squares \cite{cohen2017optimal}. The sampling process is modeled via continuous linear functionals—inner products against the elements of a Riesz basis or a frame—rather than pointwise evaluations. The acquired data are the random measurements $\<f,s_{i_1}\>,\dots,\<f,s_{i_m}\>$, which are subsequently utilized in a weighted least-squares problem. The resulting optimal coefficients synthesize the recovered signal $\tilde f$ within the finite-dimensional subspace $\mathcal W_n:=\Span\{w_1,\dots,w_n\}$.

To quantify the geometric compatibility between the sampling and reconstruction systems, we introduce the coherence parameter $R:=\sup_{j\in\N}\frac{1}{p_j}\sum_{k=1}^n|\<w_k,s_j\>|^2$. One of the main results of this work can be summarized as follows:

\begin{mainresult}
    Let $f \in \mathcal{H}$ be the unknown element we wish to reconstruct. Suppose $\{s_k\}_{k\in\N}$ is a frame, and $\{w_k\}_{k\in\N}$ is a Riesz basis such that the nested subspace condition $\overline{\Span\{w_k\st k\in\N\}}\subseteq\overline{\Span\{s_k\st k\in\N\}}$ holds. If the number of random measurements $m$ satisfies
    \[
    m\gtrsim R\log(2n/\delta),
    \]
    then with probability at least $1-\delta$, the weighted least-squares reconstruction $\tilde f$ satisfies the error bound:
    \[
    \|f-\tilde f\|\le\mu_{n,\Omega}\|P_{\mathcal W_n^\perp}f\|,
    \]
    where $\mu_{n,\Omega} \ge 1$ is an amplification factor depending only on the reconstruction dimension $n$ and the random sampling indices $\Omega=\{i_1,\dots,i_m\}$.
\end{mainresult}

This is a consequence of a more general result, Theorem \ref{mainresult}, which uses weaker geometric assumptions but yields a slightly less explicit sampling rate.

Because the reconstruction $\tilde f$ is constrained to $\mathcal W_n$, the orthogonal projection $P_{\mathcal W_n}f$ represents the best possible theoretical approximation. This theorem thus guarantees that our randomized method is robust: the reconstruction error is bounded by a small multiple of this optimal, out-of-space tail error $\|P_{\mathcal W_n^\perp}f\| = \|f - P_{\mathcal W_n}f\|$.

A major advantage of this stochastic framework is its flexibility regarding the probability distribution $p$. This allows us to actively select the distribution that minimizes the sample complexity, commonly referred to as the \emph{leverage-score distribution} \cite{leverage}. When orthonormal bases are employed, leverage scores are intimately tied to the concept of Christoffel sampling \cite{NEVAI19863,XU1995205}. We demonstrate that, even in the broadly generalized scenario where both systems are arbitrary frames, our leverage-score distribution retains a deep, explicit connection to the classical Christoffel function. 

Crucially, by implementing leverage-score sampling, we achieve the stochastic sampling rate:
\[
m\gtrsim n\log(2n/\delta).
\]
In stark contrast to deterministic approaches, this near-linear rate is \textit{universal} within our structural setting. Provided the subspace condition holds, this $n \log n$ complexity is completely independent of the specific choice of the frame $\{s_k\}_{k\in\N}$ and Riesz basis $\{w_k\}_{k\in\N}$. Furthermore, when orthonormal bases are employed, the framework yields tighter constants and precise control over the approximation error. Specifically, we prove that the factor $\mu_{n,\Omega}$ converges in probability to $1$ as $m\to+\infty$, achieving the theoretical lower bound. This effectively establishes the quasi-optimality of our method under the framework developed by Adcock, Hansen, and Poon \cite{adcock2013beyond,adcock2014optimal}. 

We demonstrate the practical power of this method on the classical problem of recovering an analytic function $f$ from $m$ continuous Fourier measurements \cite{adcock2015generalized,hrycak2010pseudospectral,demanet2019stable}. We obtain this by explicitly computing the leverage-score distribution associated with the Fourier-Legendre coherence. This computation relies on a known closed-form expression for the Fourier coefficients of Legendre polynomials via spherical Bessel functions. Consequently, our stochastic method natively overcomes the deterministic $n^2$ bottleneck, yielding near-exponential convergence rates:
\[
\|f-\tilde f\|\lesssim \sqrt m\,\rho^{-\frac{m}{\log(m/\delta)}}.
\]

Let us summarize our main contributions:
\begin{itemize}
    \item We establish a rigorous randomized framework for generalized sampling that accommodates arbitrary frames for both the sampling and reconstruction systems. Within this broad setting, we prove that whenever the reconstruction system is a Riesz basis and its span embeds into the sampling space, the near-linear rate $m\gtrsim n\log n$ guarantees stable recovery with high probability.
    \item We provide explicit concentration bounds for both the empirical Gram matrix and the empirical cross-term operator. To effectively control the aliasing error governed by the cross-term, we prove a novel variant of the matrix Bernstein inequality—extending the classical result to random rectangular operators acting between distinct, infinite-dimensional Hilbert spaces—which constitutes a theoretical contribution of independent interest.
    \item We rigorously connect our leverage-score distribution to classical Christoffel sampling in the orthonormal setting, and present a comprehensive comparative analysis between the two strategies in the more general frame setting.
    \item We showcase the efficacy of our method through the classical Fourier-Legendre analytic continuation problem: our stochastic framework allows us to achieve near-exponential convergence rates.
\end{itemize}

\subsection{Related Work}\label{subsec:related_work}

The theoretical framework developed in this work intersects with numerous active areas of research.

\textit{Generalized Sampling.}
The formalization of GS in arbitrary infinite-dimensional, separable Hilbert spaces was driven by the seminal works of Adcock and Hansen \cite{adcock2012generalized}, where instead of considering pointwise samples they resorted to abstract continuous linear functionals over Riesz bases. The main contribution of Adcock and Hansen lies not only in the nature of the measurements, but in the decoupling of the sampling and reconstruction dimensions: more specifically, they fixed the reconstruction dimension $n$ and treated the sampling dimension $m\ge n$ as a variable parameter. This strategy established the foundation for investigating \textit{stable sampling rates}: the goal is determining the minimal oversampling required to ensure numerical stability in this redundant framework. This theory was subsequently generalized to frames \cite{adcock2012frames,adcock2020frames} as well as to nonlinear, ill-posed inverse problems in \cite{alberti2022infinite}. However, as demonstrated in these foundational studies, ensuring universally stable reconstructions typically demands a stringent quadratic scaling $m\gtrsim n^2$, which often becomes computationally unfeasible in high-resolution applications. In parallel, it was shown in \cite{platte2011impossibility} that employing deterministic, equispaced grids of pointwise measurements to ensure exponential convergence in the approximation of analytic functions leads to severe numerical instability. A continuous analogue of this impossibility theorem was later established in \cite{adcock2014stability}.

\textit{Stochastic Least Squares.}
In order to overcome the aforementioned deterministic barriers, the contemporary approximation theory has been heading towards randomized strategies. Our work is strongly influenced by the seminal paper of Cohen and Migliorati \cite{cohen2017optimal}, who derived an optimal sampling strategy via weighted least-squares for the recovery of bounded functions from pointwise measurements, ultimately achieving the same log-linear sampling rate that we establish here. Other works that implement weighted least-squares methods to derive a linear rate up to log factors include \cite{kammerer2021worst}. These approaches fundamentally rely on point evaluations drawn according to the so-called Christoffel sampling strategy from a domain $D\subseteq\R^d$ endowed with a probability measure and operate with orthonormal bases in Reproducing Kernel Hilbert Spaces, a choice that makes the (reciprocal) Christoffel function particularly easy to compute. In contrast, our paper extends this paradigm to the recovery of arbitrary elements in general separable Hilbert spaces from abstract continuous linear measurements. A recent review \cite{adcock2025optimal} highlights this broader framework as a compelling direction for future research. Here, we move beyond conceptual discussion to rigorously establish the validity of the sought-after log-linear sampling rate. Specifically, we derive concentration estimates for both the random Gram matrix and the empirical cross-term operator, extending the results of aforementioned works from standard orthonormal settings to arbitrary frames.

\textit{Other Stochastic Approaches.}
The stochastic approach considered in the present work has been explored in different contexts and variants.
For instance, in order to achieve the much-coveted linear sampling rate $m\gtrsim n$, numerous endeavors have been made to develop a randomized sampling strategy that improves the sampling rate yielded by Christoffel sampling. An example of such methods are determinantal point processes \cite{determinantal}: in the context of recovering functions in a reproducing kernel Hilbert space from pointwise measurements, determinantal sampling extends the Christoﬀel strategy by introducing a kernel-dependent correlation between the sampled nodes, effectively forcing those nodes to repel each other with a repulsion related to the smoothness of the target function. Other works that aim at removing the logarithmic factor replace i.i.d. sampling with greedy procedures \cite{dolbeault2024randomized} that continuously update the Gram matrix at every sampling step, ultimately establishing a linear sampling rate at the cost of higher (polynomial) computational complexity.

Another prominent field that makes use of stochastic techniques is compressed sensing (CS)\cite{candes2006robust,donoho2006compressed}. Although it originally emerged in a finite-dimensional setting, over the past two decades an increasing number of works have been seeking a method to recover infinite-dimensional signals which are assumed to be sparse or compressible with respect to some suitable basis, which is a reasonable hypothesis in many applications. The series of works \cite{adcock2014generalized,adcock2016generalized,adcock2017breaking,adcock2018infinite,adcock2024efficient} extended CS to infinite dimension, by combining GS with existing CS tools, namely randomization and weighted $\ell^1$-minimization. This made it possible to achieve stochastic sampling rates of the form
\[
m\gtrsim \mu\cdot n\cdot s\times\text{log factors},
\]
where $s$ is the sparsity of the unknown signal $f$, the factor $\mu$ measures the coherence between the sampling and the reconstruction bases and $n > s$ is the reconstruction dimension. These ideas were recently further expanded and applied to ill-posed inverse problems in \cite{alberti2021infinite,alberti2025compressed,alberti2025compressed2}.

\subsection{Structure of the Paper}
In Section~\ref{sec:notation}, we fix our notation and present to the reader more background regarding GS. In particular, we explore the core ideas behind the deterministic scenario in order to properly generalize them to the randomized framework. Section~\ref{sec:mainresult} is devoted to presenting the main result of the paper. We establish rigorous concentration inequalities for both the empirical Gram matrix and the empirical cross-term operator. We derive the stochastic sampling rate $m\gtrsim n\log(2n/\delta)$ by drawing the samples according to a leverage-score distribution and compare this strategy with Christoffel sampling. Finally, we prove how to extend the aforementioned results to the case where the reconstruction system is an arbitrary frame. In Section~\ref{sec:analytic.func}, we present one application of our method: the reconstruction of an analytic function from Fourier measurements. Section~\ref{sec:proofs} gathers the proofs of the statements of the previous sections (with many technical lemmas postponed to Appendix~\ref{sec:appendix}). Finally, Section~\ref{sec:conclusions} presents concluding remarks and some open questions on the topic.

\section{Setup}

\label{sec:notation}
We begin by setting up our notation. 
Let $\ell^2=\ell^2(\N)$ be the Hilbert space of square summable sequences. In our convention $0\notin\N$. Let $\{e_n\}_{n\in\N}$ be the canonical basis of $\ell^2$ and let $\{\mathbf{e_1},\dots,\mathbf{e_n}\}$ denote the canonical basis of $\C^n$. In what follows, $P_n$ will denote the orthogonal projection of $\ell^2$ onto $\Span\{e_1,\dots,e_n\}$ and $P_n^\perp:=I_{\ell^2}-P_n$. We also define $\iota_n\colon\C^n\longrightarrow\ell^2$ by $\iota_n(x_1,\dots,x_n)=(x_1,\dots,x_n,0,0,\dots)$. It's easy to see that $\iota_n^*\colon\ell^2\longrightarrow\C^n$ cuts an infinite sequence $(x_n)_{n\in\N}\in\ell^2$ into its first $n$ components $(x_1,\dots,x_n)\in\C^n$. We have the identities $\iota_n^*\iota_n=I_{\C^n}$ and $\iota_n\iota_n^*=P_n$. For an operator $A$, we will denote its range by $\ran(A)$.

Fix a separable Hilbert space $\mathcal H$. Recall the definitions of Riesz sequence, Riesz basis and frame.
\begin{definition}
    A sequence $(v_k)_{k\in\N}$ of $\mathcal H$ is called a Riesz sequence if there exists two positive constants $A,B>0$ such that for every $x\in\ell^2$ one has
    \[
    A\|x\|^2\le\left\|\sum_{k\in\N}x_kv_k\right\|^2\le B\|x\|^2.
    \]
    Equivalently, the synthesis operator associated to $(v_k)_{k\in\N}$:
    \[
    x\in\ell^2\longmapsto\sum_{k\in\N}x_kv_k\in\mathcal H
    \]
    is bounded, injective and has closed range. The Riesz sequence $(v_k)_{k\in\N}$ is a Riesz basis for $\mathcal H$ if
    \[
    \overline{\operatorname{span}\{v_k\st k\in\N\}}=\mathcal H,
    \]
    that is, if the synthesis operator associated to $(v_n)_{n\in\N}$ is bounded and bijective.
\end{definition}
\begin{definition}
    A sequence $(v_k)_{k\in\N}$ of $\mathcal H$ is called a frame if there exist two positive constants $\alpha,\beta>0$ such that for every $h\in\mathcal H$ one has
    \[
    \alpha\|h\|^2\le\sum_{k\in\N}|\<h,v_k\>|^2\le\beta\|h\|^2.
    \]
    This is equivalent to requiring the associated analysis operator
    \[
    h\in\mathcal H\longmapsto (\<h,v_k\>)_{k\in\N}\in\ell^2
    \]
    to be bounded and injective with closed range, which in turn means that the associated synthesis operator is bounded and surjective.
\end{definition}
The last part of the previous definitions make it clear that a Riesz basis is in particular a frame. 
Throughout this paper, we often write "Riesz basis" to denote a Riesz sequence in the ambient space $\mathcal H$, with the implicit understanding that it constitutes a Riesz basis for its closed linear span. The analogous convention holds for the word "frame".

\subsection{Background on Generalized Sampling}
Let $f\in\mathcal H$ be the unknown element we would like to recover. Suppose we are given a frame $\{s_k\}_{k\in\N}$ (sampling vectors) and a Riesz sequence $\{w_k\}_{k\in\N}$ (reconstruction vectors), and let $\mathcal S:=\overline{\Span\{s_k\st k\in\N\}}$ and $\mathcal W:=\overline{\Span\{w_k\st k\in\N\}}$. If $n\in\N$, we set $\mathcal S_n:=\Span\{s_1,\dots,s_n\}$ and $\mathcal W_n:=\Span\{w_1,\dots,w_n\}$. Each sequence has its associated synthesis operator:
\begin{equation*}
    \begin{aligned}
        S\colon\ell^2 &\longrightarrow \mathcal H, &\qquad
        W\colon\ell^2 &\longrightarrow \mathcal H,\\
        x &\longmapsto \sum_{k\in\N} x_k s_k, &\qquad
        x &\longmapsto \sum_{k\in\N} x_k w_k.
    \end{aligned}
\end{equation*}
The analysis operators associated to $\{s_k\}_{k\in\N}$ and $\{w_k\}_{k\in\N}$ are $S^*$ and $W^*$ respectively, which are defined by
\begin{equation*}
    \begin{aligned}
        S^*\colon\mathcal H&\longrightarrow\ell^2, &\qquad
        W^*\colon\mathcal H&\longrightarrow\ell^2,\\
        h &\longmapsto (\<h,s_k\>)_{k\in\N}, &\qquad
        h &\longmapsto (\<h,w_k\>)_{k\in\N}.
    \end{aligned}
\end{equation*}
By definition, there exist positive constants $A,B,C,D>0$ such that for every $x\in\ell^2$ and every $s\in\mathcal S$ we have
\begin{equation}
    A\|s\|^2\le\|S^*s\|^2\le B\|s\|^2,\qquad C\|x\|^2\le\|Wx\|^2\le D\|x\|^2.
\end{equation}
It's known that any Riesz basis for a Hilbert space is a frame for that space and the frame bounds coincide with the Riesz basis bounds (cf. \cite[Theorem 5.4.1]{christensen2003introduction}). It follows that for every $w\in\mathcal W$
\[
C\|w\|^2\le\|W^*w\|^2\le D\|w\|^2.
\]
In particular, both $S$ and $W$ are bounded with $\|S\|\le\sqrt B$, $\|W\|\le\sqrt D$. Moreover, $W$ is a bijection between $\ell^2$ and $\mathcal W$. On the other hand, $S(\ell^2)=\mathcal S$ but $S$ may not be one-to-one.

Our goal is to obtain a reconstruction $\tilde f\in\mathcal W$ of $f$ based on the sampling vectors $S^*f=\{\langle f,s_k\rangle\}_{k\in\N}$. If we assume that $\mathcal H=\mathcal W\oplus\mathcal S^\perp$ and that also $\{s_k\}_{k\in\N}$ is a Riesz basis for $\mathcal S$, a natural choice would be
\[
\tilde f:=W\tilde\beta,
\]
where $\tilde\beta\in\operatorname{argmin}\frac{1}{2}\|S^*W\beta-S^*f\|^2$. This can be solved explicitly: 
\[
\tilde f=W(S^*W)^{-1}S^*f.
\]
Indeed, under the above assumptions, the operator $U:=S^*W$ is invertible and $WU^{-1}S^*$ is precisely the projection onto $\mathcal W$ along $\mathcal S^\perp$.
\smallskip

However, in practice, one does not have the infinite amount of samples $S^*f\in\ell^2$ at disposal. In \cite{adcock2012generalized}, Adcock and Hansen fixed $n\in\N$ (the reconstruction dimension) and proved (under the assumption $\mathcal W\cap\mathcal S^\perp=\{0\}$) that for sufficiently large $m\in\N$ (the sampling dimension) the $n\times n$ matrix $\iota_n^*U^*P_mU\iota_n$ is invertible, allowing them to deduce that the problem
\[
P_mU\iota_n x=P_mS^*f
\]
admits a unique least-squares solution given by $\tilde x=(\iota_n^*U^*P_mU\iota_n)^{-1}\iota_n^*U^*P_mS^*f$. The reconstruction of $f\in\mathcal H$ is then performed in the following way:
\[
\tilde f:=W\iota_n\tilde x=W\iota_n(\iota_n^*U^*P_mU\iota_n)^{-1}\iota_n^*U^*P_mS^*f\in\mathcal W_n,
\]
with an estimate on the error of the form
\[
\|f-\tilde f\|\lesssim(1+\|\iota_n(\iota_n^*U^*P_mU\iota_n)^{-1}\iota_n^*U^*P_mUP_n^\perp\|)\|P_n^\perp\beta\|.
\]

\subsection{Stochastic Generalized Sampling}

Rather than taking the first $m$ deterministic samples from $\mathbf y:=S^*f=(\<f,s_k\>)_{k\in\N}$, we propose selecting $m$ indices $\Omega=\{i_1,\dots,i_m\}\subseteq\N$ at random according to a suitable probability distribution $p=(p_j)_{j\in\N}$. This randomized sampling strategy ensures the stable recovery of $f$ with high probability using a much smaller sampling budget $m$.

To formalize this, let $i_1,\dots,i_m$ be drawn independently from $p$. We define the weighted random projection $Q_{\Omega}$ onto $\Span\{e_{i_1},\dots,e_{i_m}\}$ as
\begin{equation*}
    Q_{\Omega}:=\frac{1}{m}\sum_{t=1}^m\frac{e_{i_t}\otimes e_{i_t}}{p_{i_t}},
\end{equation*}
which is well-defined almost surely since indices are drawn according to $p$ (and thus $p_{i_t} > 0$ with probability one). 

Our goal is to obtain a reconstruction $\tilde f:=W\iota_n\tilde x$, where the coefficients $\tilde x\in\C^n$ minimize the discrepancy between the acquired samples and the reconstructed samples. To prevent highly probable indices from unfairly dominating the target element, we solve the weighted least-squares problem:
\[
\underset{x\in\mathbb{C}^n}{\operatorname{argmin}}\sum_{t=1}^m\frac{1}{p_{i_t}}|\<W\iota_nx,s_{i_t}\>-\<f,s_{i_t}\>|^2.
\]
Observe that this objective function can be rewritten in operator form:
\begin{align*}
    \frac{1}{m}\sum_{t=1}^m\frac{1}{p_{i_t}}|\<W\iota_nx,s_{i_t}\>-\<f,s_{i_t}\>|^2&=\frac{1}{m}\sum_{t=1}^m\frac{1}{p_{i_t}}|\<U\iota_nx-\mathbf y,e_{i_t}\>|^2\\
    &=\frac{1}{m}\sum_{t=1}^m\frac{1}{p_{i_t}}\<(e_{i_t}\otimes e_{i_t})(U\iota_nx-\mathbf y),U\iota_nx-\mathbf y\>\\
    &=\left\|\sqrt{Q_\Omega}(U\iota_n x-\mathbf y)\right\|^2.
\end{align*}
We are thus looking for a least-squares solution $\tilde x$ to the linear system
\begin{align}\label{problem}
    \sqrt{Q_\Omega}U\iota_n x=\sqrt{Q_\Omega}\mathbf y.
\end{align}
Such a solution exists and is unique if the random Gram matrix 
\[
\randgram:=(\sqrt{Q_\Omega}U\iota_n)^*\sqrt{Q_\Omega}U\iota_n=\iota_n^*U^*Q_\Omega U\iota_n
\]
is invertible. By construction, $\randgram$ is an unbiased estimator of the $n\times n$ finite section $\Sigma:=\iota_n^*U^*U\iota_n$. Inspired by the deterministic scenario, we will prove that $\randgram$ is eventually invertible with high probability by showing that it converges in the operator norm to $\Sigma$. (Under the assumption $\mathcal W\cap\mathcal S^\perp=\{0\}$, the limiting matrix $\Sigma$ is strictly positive definite and invertible, as shown in Proposition \ref{sigma.invertible}).

To rigorously analyze this convergence, we introduce the following key quantities. Let $\lambda_0>0$ denote the smallest eigenvalue of $\Sigma$, so that $\|\Sigma^{-1}\|=\frac{1}{\lambda_0}$. For each frequency index $j\in\N$, we isolate the specific finite-dimensional interaction vector $v_j:=\iota_n^*U^*e_j\in\C^n$, and we define the coherence bound:
\[
R:=\sup_{j\in\supp(p)}\frac{\|v_j\|^2}{p_j}.
\]

In addition to the Gram matrix, our reconstruction error analysis requires us to understand how the out-of-space tail of the signal (i.e., $P_{\mathcal W_n^\perp}f$) aliases into our finite-dimensional approximation. This aliasing is governed by the empirical cross-term operator:
\[
\widehat C_\Omega:=\iota_n^*U^*Q_{\Omega}S^*P_{\mathcal W_n^\perp}\colon\mathcal H\longrightarrow\C^n.
\]
Analogous to the Gram matrix, $\widehat C_\Omega$ acts as an unbiased estimator for the true cross-term operator $C:=\iota_n^*U^*S^*P_{\mathcal W_n^\perp}$. To control the concentration of this operator, we define the residual vectors $u_j:=P_{\mathcal W_n^\perp}s_j\in\mathcal W_n^\perp$ along with their respective bound $R':=\sup_{j\in\supp(p)}\frac{\|u_j\|^2}{p_j}$.

To unify our subsequent concentration inequalities, we define the joint coherence bound $R'':=\max\{R,R'\}$. Assuming that $R'<+\infty$ (see Assumption \ref{assumptions}), the series $\sum_{j\in\supp(p)}u_j\otimes u_j$ converges absolutely in the operator norm to a bounded operator $T\in\mathcal B(\mathcal H)$. Finally, we let $K:=\max\{\|\Sigma\|,\|T\|\}$ denote the maximum scale of our limiting operators.

\section{Main result}\label{sec:mainresult}
Throughout this work, we operate under the following assumptions.

\begin{assumption}\label{assumptions}
    \begin{enumerate}[i)]
        \item The reconstruction space trivially intersects the orthogonal complement of the sampling space:
        \[
        \mathcal W\cap\mathcal S^\perp=\{0\}.
        \]
        (Note that this condition is naturally satisfied in many practical settings, such as when $\mathcal W \subseteq \mathcal S$).
        \item The probability distribution $p$ yields finite coherence bounds $R$ and $R'$ (and consequently, $R'' < +\infty$).
        \item The support of the probability distribution $p$ contains the support of the sequence $(v_j)_{j\in\N}$. Equivalently, $p_j > 0$ whenever $v_j \neq 0$:
        \begin{equation}\label{supp(p)}
            \supp(p):=\{j\in\N\st p_j > 0\}\supseteq\{j\in\N\st v_j\neq0\}.
        \end{equation}
    \end{enumerate}
\end{assumption}

\subsection{Main Result}\label{subsec:riesz_and_onb}

We first present our main result for the case where $\{s_k\}_{k\in\N}$ is a frame and $\{w_k\}_{k\in\N}$ is a Riesz basis. We then conclude this subsection by detailing its specialization to the orthonormal case.
We recall the definitions of the weighted sampling operator $Q_\Omega := \frac{1}{m}\sum_{t=1}^m \frac{e_{i_t}\otimes e_{i_t}}{p_{i_t}}$, the empirical Gram matrix $\randgram := \iota_n^*W^*SQ_\Omega S^*W\iota_n$ and the empirical cross-term $\widehat C_\Omega := \iota_n^*W^*SQ_\Omega S^*P_{\mathcal W_n^\perp}$.

\begin{theorem}\label{mainresult}
    Let $\mathcal H$ be a separable Hilbert space. Let $\{s_k\}_{k\in\N}$ be a frame for its closed span $\mathcal S$, and let $\{w_k\}_{k\in\N}$ be a Riesz basis for its closed span $\mathcal W$, with respective synthesis operators $S$ and $W$. 
    
    Let $f\in\mathcal H$ and let $\mathcal W_n:=\Span\{w_1,\dots,w_n\}$. Suppose we draw $m$ i.i.d. sample indices $\Omega=\{i_1,\dots,i_m\}$ from $\N$ according to a probability distribution $p = (p_j)_{j\in\N}$. 
    
    Under Assumption \ref{assumptions}, if for a given $\delta\in(0,1)$ the sample size $m$ satisfies 
    \[
    m\ge\frac{8}{3} R\|\Sigma\|\|\Sigma^{-1}\|^2\log(2n/\delta),
    \]
    then with probability at least $1-\delta$, the weighted least-squares problem
    \[
    \underset{x\in\mathbb{C}^n}{\operatorname{argmin}}\sum_{t=1}^m\frac{1}{p_{i_t}}|\<W\iota_n x,s_{i_t}\>-\<f,s_{i_t}\>|^2
    \]
    admits a unique solution $\tilde x$. Moreover, the reconstructed signal $\tilde f:=W\iota_n\tilde x$ satisfies the error bound
    \[
    \|f-\tilde f\|\le\|P_{\mathcal W_n^\perp}f\|\left(1+K_{n,\Omega}^2\right)^\frac{1}{2},
    \]
    where $K_{n,\Omega}:=\|W\iota_n\randgram^{-1}\widehat C_\Omega\|$.
\end{theorem}

 To the best of our knowledge, Theorem \ref{mainresult} constitutes the first result extending the generalized sampling framework to a stochastic setting with explicit bounds on the sample complexity. While a conceptual generalization of randomized least-squares to abstract measurements was recently discussed by Adcock \cite[Section 9]{adcock2025optimal}, the analysis therein does not provide explicit sampling bounds for the discrete frame interactions. Crucially, it does not address the aliasing error introduced by the infinite-dimensional basis mismatch. In our framework, this mismatch is governed by the empirical cross-term operator $\widehat C_\Omega$, whose concentration properties are rigorously established in Section \ref{subsec:crossterm} to guarantee stable recovery. 

Under slightly stronger assumptions, we can remove the dependence on $\|\Sigma\|$ and $\|\Sigma^{-1}\|$ in the sample complexity bounds.

\begin{corollary}\label{cor:Rlogn}
    Consider the setting of Theorem \ref{mainresult} and suppose further that $\mathcal W\subseteq\mathcal S$. Then $\|\Sigma\|\le BD$ and $\|\Sigma^{-1}\|\le\frac{1}{AC}$. In particular, the conclusions of Theorem \ref{mainresult} hold provided that
    \begin{equation}\label{Rlog n}
        m\gtrsim R\log(2n/\delta).
    \end{equation}
\end{corollary}
 
Recall that, although suppressed in the notation, the parameter $R$ inherently depends on the reconstruction dimension $n$, due to the fact that each $v_j$ belongs to $\C^n$. We show in Section~\ref{subsec:leverage_score} that, starting from the new rate \eqref{Rlog n}, one can then choose an appropriate probability distribution $p$ such that $R$ grows (at most) linearly with $n$, ultimately resulting in the desired sampling rate $m\gtrsim n\log(2n/\delta)$. 

While one might hope for an even slower growth rate, linear scaling is in fact the theoretical limit. Indeed, as shown in Proposition \ref{trace}, the trace of $\Sigma$ can be bounded by $R$:
\[
\tr(\Sigma)=\sum_{j\in\N}\|v_j\|^2\le R\sum_{j\in\N}p_j=R.
\]
This inequality, combined with the fact that $\tr(\Sigma)\ge\lambda_0n$, proves that $R$ must grow at least linearly with $n$. This constraint aligns with contemporary literature (e.g., \cite{adcock2025optimal}) and is intrinsically tied to Christoffel sampling, a concept we will explore further in Section~\ref{subsec:leverage_score}.

\smallskip
\textit{Proof Strategy.} The proof of Theorem \ref{mainresult} proceeds in three main steps, detailed fully in Section \ref{sec:proofs}. First, we express the empirical Gram matrix error $\randgram-\Sigma$ as a sum of independent, centered, self-adjoint random matrices. By applying the matrix Bernstein inequality (Theorem \ref{bernstein}) and carefully bounding the variance and spectral norm of the summands (Proposition \ref{prob.conv}), we deduce that $\|\randgram-\Sigma\| < \lambda_0$ with high probability, provided $m$ is sufficiently large. Second, a standard perturbation argument ensures that $\randgram$ is invertible on this high-probability event, guaranteeing the existence and uniqueness of the least-squares solution. Finally, to estimate the error, we decompose the true signal as $f = P_{\mathcal W_n}f + P_{\mathcal W_n^\perp}f$ and explicitly track the aliasing of the out-of-space tail $P_{\mathcal W_n^\perp}f$ through the inversion process. This algebraic manipulation isolates the empirical cross-term $\widehat C_\Omega$, yielding the final bound.

\subsubsection{Orthonormal Bases}

Suppose now that we are given orthonormal bases $\{s_k\}_{k\in\N}$ and $\{w_k\}_{k\in\N}$ for $\mathcal S$ and $\mathcal W$ respectively. Suppose further that $\mathcal W\subseteq\mathcal S$: under these additional assumptions, the synthesis operator $S,W$ become isometries, and so does $U$:
\[
U^*U=W^*SS^*W=W^*P_{\mathcal S}W=W^*W=I_{\ell^2}.
\]
But then $\Sigma=\iota_n^*U^*U\iota_n=\iota_n^*\iota_n=I_{\C^n}$ is simply the identity $n\times n$ matrix, hence $\|\Sigma\|=\lambda_0=1$. We thus deduce immediately from Theorem~\ref{mainresult} the following result.
\begin{corollary}\label{full.ON.case}
    Consider the setting of Theorem \ref{mainresult}. Assume further that $\mathcal W\subseteq\mathcal S$ and that both $\{s_k\}_{k\in\N}$ and $\{w_k\}_{k\in\N}$ are orthonormal bases for $\mathcal S,\mathcal W$ respectively. Then the conclusions of Theorem \ref{mainresult} hold provided that $m\ge\frac{8}{3}R\log(2n/\delta)$.
\end{corollary}

This simplified orthonormal setting is frequently encountered in practice. Some of the most notable examples include the reconstruction of wavelet coefficients from Fourier measurements (the Fourier-Wavelet setup) and the approximation of analytic functions via Legendre polynomials from Fourier data. We will extensively study this Fourier-Legendre setup in Section \ref{sec:analytic.func} as a primary application of our theory.

\subsection{Concentration of the Empirical Cross-Term}\label{subsec:crossterm}

In Theorem \ref{mainresult}, the reconstruction error is bounded by a multiple of the theoretical best-approximation error, scaled by the factor $(1+K_{n,\Omega}^2)^{1/2}$ where $K_{n,\Omega}=\|W\iota_n\randgram^{-1}\widehat C_\Omega\|$. Clearly, such an estimate is of practical significance only if we can establish rigorous control over $K_{n,\Omega}$ as the sample size $m$ grows. In the deterministic Generalized Sampling framework, the corresponding cross-term parameter is known to converge to the deterministic limit $\|W\iota_n\Sigma^{-1}C\|$ \cite{adcock2012stable,adcock2012frames}. 

Motivated by this, our primary goal in this subsection is to prove that $K_{n,\Omega}$ converges in probability to this same limit, and to provide an explicit concentration inequality for this convergence. To establish this, we rely on the following algebraic decomposition:
\begin{equation}\label{eq.K_n,omega}
    |K_{n,\Omega}-\|W\iota_n\Sigma^{-1}C\|| = \left| \|W\iota_n\randgram^{-1}\widehat C_\Omega\|-\|W\iota_n\Sigma^{-1}C\| \right| \le\sqrt D\|\randgram^{-1}\widehat C_\Omega-\Sigma^{-1}C\|.
\end{equation}
Since we have already established that $\randgram \xrightarrow[m\to+\infty]{\P}\Sigma$, a standard continuity argument ensures $\randgram^{-1}\xrightarrow[m\to+\infty]{\P}\Sigma^{-1}$. Therefore, it suffices to prove that the empirical cross-term concentrates around its expectation ($\widehat C_\Omega\xrightarrow[m\to+\infty]{\P}C$) with a precise probability threshold $\tilde M_{\varepsilon,\delta}$, and then invoke the continuity of operator composition. 

However, analyzing $\widehat C_\Omega$ introduces specific technical challenges. Unlike the Gram matrix, $\widehat C_\Omega \colon \mathcal H \longrightarrow \C^n$ is a rectangular operator mapping between two distinct Hilbert spaces, one infinite dimensional, the other finite dimensional. To overcome this, we develop a novel extension of the Bernstein inequality. By combining the Hermitian dilation technique typically used for finite rectangular matrices with a dimension-free Bernstein inequality for Hilbert-Schmidt operators (Theorem \ref{op.bernstein}), we obtain the following general result.

\begin{theorem}[Bernstein inequality for random, rectangular operators]\label{op.rect.bernstein}
    Let $\mathcal H_1,\mathcal H_2$ be separable Hilbert spaces and let $X_1,\dots,X_m\colon\mathcal H_1\longrightarrow\mathcal H_2$ be independent, random operators such that $\E(X_t)=0$ for every $t\in\{1,\dots,m\}$. Assume that there exists $L,\sigma^2\ge0$ such that for every $t\in\{1,\dots,m\}$ one has:
    \begin{align*}
    &\|X_t\|\le L\quad\text{almost surely},\\
    &\max\{\|V_1\|,\|V_2\|\}\le\sigma^2,
    \end{align*}
    where $V_1:=\sum_{t=1}^m\E(X_t^*X_t)$ and $V_2:=\sum_{t=1}^m\E(X_tX_t^*)$. Finally, suppose that for every $t\in\{1,\dots,m\}$ the operators $X_t^*X_t$ and $X_tX_t^*$ are trace-class on $\mathcal H_1$ and $\mathcal H_2$ respectively. Then for every $\varepsilon\ge\frac{1}{6}(L+\sqrt{L^2+36\sigma^2})$ one has
    \[
    \P\left(\left\|\sum_{t=1}^m X_t\right\|\ge\varepsilon\right)\le 14\frac{\tr(V_1)+\tr(V_2)}{\max\{\|V_1\|,\|V_2\|\}}\exp\left(\frac{-\varepsilon^2/2}{\sigma^2+L\varepsilon/3}\right).
    \]
\end{theorem}

Although the proof fundamentally relies on combining two well-established variants of the Bernstein inequality, to the best of our knowledge, this specific formulation for random, rectangular operators between abstract Hilbert spaces does not appear in the existing literature. The full proof is deferred to Section \ref{subsec:proofs_concentration}.

With this tool at our disposal, we can now establish the concentration inequality for the cross-term factor $K_{n,\Omega}$.

\begin{theorem}\label{K_n,omega}
    Consider the setting of Theorem \ref{mainresult}. Let $\Lambda:=1+\|\Sigma^{-1}\|+\|C\|$ and let $\varepsilon\in(0,\Lambda\sqrt D\min\{1,2\|\Sigma^{-1}\|\})$ and $\delta\in(0,1)$. If
    \[
    m\ge M_{\varepsilon,\delta}^*:=\frac{10}{3}KR''\log(56n/\delta)\frac{\Lambda^2D}{\varepsilon^2}\max\left\{1,2\|\Sigma^{-1}\|\right\}^4,
    \]
    then
    \[
    \left|K_{n,\Omega}-\|W\iota_n\Sigma^{-1}C\|\right|<\varepsilon
    \]
    with probability at least $1-\delta$.
\end{theorem}

If we specialize this result to the orthonormal case and assume geometric compatibility between the spaces, the cross-term behavior simplifies dramatically. We are able to prove that the cross-term vanishes entirely by only requiring the sampling system $\{s_k\}_{k\in\N}$ to be orthonormal.

\begin{corollary}\label{half.ONB.case}
     If $\{s_k\}_{k\in\N}$ is an orthonormal basis for $\mathcal S$ and $\mathcal W\subseteq\mathcal S$, then $C=0$, hence $K_{n,\Omega}\xrightarrow[m\to+\infty]{\P}0$.
\end{corollary}

This corollary highlights a remarkable feature of orthonormal sampling systems: the aliasing error caused by the out-of-space tail of the signal is completely eliminated in the limit. Because $K_{n,\Omega} \xrightarrow{\P} 0$, our randomized reconstruction asymptotically achieves the theoretical best-approximation error $\|f-\tilde f\| \le \|P_{\mathcal W_n^\perp}f\|$ without suffering from any noise amplification from the un-reconstructed tail.

\subsection{Leverage-score Sampling}\label{subsec:leverage_score}

Up to this point, our stochastic framework has accommodated any arbitrary probability distribution $p$ satisfying Assumption \ref{assumptions}. We now specialize our analysis to the optimal choice of measure, which minimizes the coherence parameter $R$ and consequently yields the most efficient sampling rate.

Recall from Proposition \ref{trace} that the trace of the matrix $\Sigma$ is exactly $\sum_{j\in\N}\|v_j\|^2$. We can thus define the following probability distribution on $\N$:
\[
p_j:=\frac{\|v_j\|^2}{\tr(\Sigma)}\quad\text{for every }j\in\N.
\]
A randomized sampling strategy that employs this specific distribution will be referred to as \textit{leverage-score sampling}. Note that this density naturally satisfies Assumption \ref{assumptions}: clearly, $v_j\neq0$ forces $p_j>0$. Moreover, this optimal choice yields:
\[
R=\sup_{j\in\supp(p)}\frac{\|v_j\|^2}{p_j}=\tr(\Sigma)<+\infty.
\]
Substituting this value of $R$ into our main theorem immediately yields the following result.
\begin{corollary}\label{leverage}
    Consider the setting of Theorem \ref{mainresult} and choose $p_j=\frac{\|v_j\|^2}{\tr(\Sigma)}$ for every $j\in\N$. Then the conclusions of Theorem \ref{mainresult} hold provided that $m\ge\frac{8}{3}\tr(\Sigma)\|\Sigma\|\|\Sigma^{-1}\|^2\log(2n/\delta)$.
\end{corollary}
We remark that, since $\tr(\Sigma)\le n\|\Sigma\|$, the stable sampling rate with the leverage-score choice can be safely bounded by
\[
m\ge\frac{8}{3}\kappa(\Sigma)^2n\log(2n/\delta),
\]
where $\kappa(\Sigma)=\|\Sigma\|\|\Sigma^{-1}\|$ is the condition number of $\Sigma$. Under the additional subspace condition $\mathcal W\subseteq\mathcal S$, we know from Proposition \ref{unifbound.cond_number} that this condition number is controlled uniformly in $n$ by the frame/Riesz bounds. We thus obtain the following universal rate.

\begin{corollary}\label{cor:leverage+subspace.condition}
    Consider the setting of Theorem \ref{mainresult} and choose $p_j=\frac{\|v_j\|^2}{\tr(\Sigma)}$ for every $j\in\N$. Assume further that $\mathcal W\subseteq\mathcal S$. Then the conclusions of Theorem \ref{mainresult} hold provided that $m\gtrsim n\log(2n/\delta)$.
\end{corollary}

This corollary highlights a profound advantage of the stochastic framework. By employing leverage-score sampling, we achieve an explicit, near-linear sample complexity of $m \gtrsim n \log n$ that is \textit{completely independent of the specific choice of the bases}. This stands in stark contrast to deterministic generalized sampling, where the stable sampling rate is strictly basis-dependent, often impossible to compute explicitly, and frequently suffers from a severe quadratic scaling barrier ($m \gtrsim n^2$).

When orthonormal bases are employed, the asymptotic sampling rate remains $\mathcal{O}(n \log n)$, but we obtain tighter constants and, crucially, the cross-term factor $K_{n,\Omega}$ decays to zero. Indeed, if both $\{s_k\}_{k\in\N}$ and $\{w_k\}_{k\in\N}$ are orthonormal bases and $\mathcal W\subseteq\mathcal S$, then $\Sigma=I_{\C^n}$, therefore $R=\tr(I_{\C^n})=n$. We immediately derive the following.

\begin{corollary}\label{cor:onb_nlogn}
    Consider the setting of Corollary \ref{full.ON.case} and choose $p_j=\frac{\|v_j\|^2}{n}$ for every $j\in\N$. Then the conclusions of Theorem \ref{mainresult} hold provided that $m\ge\frac{8}{3}n\log(2n/\delta)$.
\end{corollary}

\subsubsection{Connections with Christoffel Sampling}

In the literature of randomized least-squares approximation for continuous functions, optimal sampling distributions are intimately tied to the concept of the \textit{Christoffel function}. Given that our leverage-score distribution $p_j \propto \|v_j\|^2$ yields the optimal $m \gtrsim n \log n$ rate in the discrete, abstract operator setting, it is natural to ask how our framework mathematically mirrors this established continuous concept. 

To see this connection, let us briefly review the continuous setting. If $(\mathcal D,\rho)$ is a measure space and $\mathcal P\subseteq L^2(\mathcal D,\rho)$ is an $n$-dimensional subspace where pointwise evaluation is defined, the (reciprocal) Christoffel function $\mathcal K_\mathcal P\colon\mathcal D\longrightarrow[0,+\infty)$ is defined such that $\mathcal K_\mathcal P(x)$ is the squared norm of the evaluation functional $p\in\mathcal P\mapsto p(x)\in\C$. (This functional is well-defined by assumption and continuous due to the finite dimensionality of $\mathcal P$). 

If $\mu$ is a probability measure on $\mathcal D$ that is absolutely continuous with respect to $\rho$, it admits a positive density $d\mu=\nu\,d\rho$. By defining the weight function $w:=\frac{1}{\nu}$, the global coherence of the subspace is captured by the parameter:
\[
\kappa_w(\mathcal P):=\|w\,\mathcal K_\mathcal P\|_{L^\infty(\mathcal D,\rho)}.
\]
It is a well-established result \cite{adcock2025optimal} that if $m$ samples $x_1,\dots,x_m\in\mathcal D$ are drawn independently according to $\mu$, one can stably recover a function $f\in L^2(\mathcal D,\rho)$ from its pointwise evaluations using a finite-dimensional reconstruction $\tilde f\in\mathcal P$ with probability at least $1-\delta$, provided that:
\[
m\gtrsim\kappa_w(\mathcal P)\log(2n/\delta).
\]
Comparing this to our stable sampling rate, it is evident that $\kappa_w(\mathcal P)$ plays the exact same role as our coherence parameter $R$. We now rigorously explore this connection both for the orthonormal setting and the frame-to-Riesz scenario.

\smallskip\noindent
\textit{Orthonormal Bases.} 
In our discrete setting, the domain is $\mathcal D=\N$ and $\rho$ is the counting measure, meaning $L^2(\mathcal D,\rho)=\ell^2$. Our $n$-dimensional subspace is $\mathcal P=U\iota_n(\C^n)=S^*(\mathcal W_n)$. The density $\nu$ corresponds to our probability distribution $p=(p_j)_{j\in\N}$ (which acts as the Radon-Nikodym derivative with respect to the counting measure), yielding the weight function $w(j)=\frac{1}{p_j}$. Assuming for simplicity that $p$ is strictly positive on $\N$, the next lemma demonstrates that in the orthonormal setting, our parameter $R$ is exactly the Christoffel parameter.

\begin{lemma}\label{christoffel.function}
    With the above notation and assuming orthonormal bases, we have:
    \[
    \kappa_w(\mathcal P)=\sup_{j\in\N}\frac{\|v_j\|^2}{p_j}=:R.
    \]
\end{lemma}
The proof of this lemma is deferred to Subsection \ref{subsec:proofs_leverage}.

\smallskip\noindent
\textit{Frame-to-Riesz Case.} 
It is natural to ask if the exact equality $\kappa_w(\mathcal P)=R$ holds in the more general Riesz basis case. To answer this, we must explicitly compute the Christoffel function $\mathcal K_\mathcal P(j)$ for our discrete subspace $\mathcal P$. 

By definition, $\mathcal K_\mathcal P(j)$ is the squared norm of the evaluation functional $a \mapsto a_j$ on $\mathcal P$. By the Riesz Representation Theorem, there exists a unique representer $q_j\in\mathcal P$ such that $\<a,q_j\>_{\ell^2}=a_j = \<a, e_j\>_{\ell^2}$ for every $a\in\mathcal P$. This orthogonality condition implies that the difference $q_j-e_j$ lies in $\mathcal P^\perp$.

Because $q_j\in\mathcal P$, taking the orthogonal projection $P_{\mathcal P}$ yields $q_j=P_{\mathcal P}e_j$. Utilizing the explicit formula for the projection onto the range of $U\iota_n$, we find:
\[
q_j = P_{\ran(U\iota_n)}e_j = U\iota_n(\iota_n^* U^* U \iota_n)^{-1}\iota_n^*U^*e_j = U\iota_n\Sigma^{-1}v_j.
\]
We can now compute the Christoffel function by evaluating the squared norm of this representer:
\begin{align}\label{christoffel_riesz}
    \mathcal K_\mathcal P(j) &= \|q_j\|_{\ell^2}^2 = \<U\iota_n\Sigma^{-1}v_j,U\iota_n\Sigma^{-1}v_j\>_{\ell^2} \nonumber \\
    &= \<\Sigma^{-1}v_j,\Sigma\Sigma^{-1}v_j\>_{\ell^2} = \<\Sigma^{-1}v_j,v_j\>_{\ell^2}.
\end{align}
Consequently, the weighted Christoffel parameter is not exactly $R$, but rather:
\[
\kappa_w(\mathcal P)=\sup_{j\in\N}\frac{\<\Sigma^{-1}v_j,v_j\>_{\ell^2}}{p_j}.
\]
By bounding this quadratic form, we see that the parameters are equivalent up to the spectral bounds of the finite section $\Sigma$:
\[
\frac{1}{\|\Sigma\|}R\le\kappa_w(\mathcal P)\le\|\Sigma^{-1}\|R.
\]

\subsection{Extension to \texorpdfstring{$\{w_k\}_{k\in\N}$}{wk} Frame}

We now consider the setting where we relax the condition on the reconstruction system $\{w_k\}_{k\in\N}$, requiring it only to be a frame for $\mathcal W$ rather than a Riesz basis.

Recall that any Riesz basis constitutes a frame with the identical bounds (cf. \cite[Theorem 5.4.1]{christensen2003introduction}). For this reason, we deliberately retain the notation $C,D$ to denote the frame bounds of the synthesis operator $W$. This guarantees that our broader frame notation naturally recovers the previously established Riesz bounds when restricted to that special case. 

This generalization has been studied in the deterministic setting by Adcock-Hansen-Herrholz-Teschke \cite{adcock2012frames}, where they developed analogous stable sampling rates to those in the Riesz scenario \cite{adcock2012generalized}. The core issue in the frame setting is that, although the convergence in probability $\randgram\xrightarrow[m\to+\infty]{\P}\Sigma$ remains true (Proposition \ref{prob.conv} does not rely on $\{w_k\}_{k\in\N}$ being a Riesz basis), we no longer have the guarantee that $\randgram$ is eventually invertible. Because the synthesis operator $W$ is only required to be bounded and surjective onto $\mathcal W$ (not necessarily injective), the limiting finite section $\Sigma$ itself need not be invertible.

To overcome this lack of invertibility, we employ the Moore-Penrose pseudo-inverse, following the approach of \cite{adcock2012frames}. The recovery $\tilde f$ is now defined as:
\begin{equation}\label{tilde.f.frames}
    \tilde f:=W\iota_n\randgram^\dagger\iota_n^*U^*Q_\Omega S^*f.
\end{equation}
In this subsection, $\lambda_0>0$ will denote the smallest \textit{non-zero} eigenvalue of $\Sigma$. With this extended notation, the minimal stable sampling rate $M_\delta=M_{\lambda_0,\delta}$ coincides with the one proposed in the Riesz scenario, essentially due to the fact that $\|\Sigma^\dagger\|=\lambda_0^{-1}$. 

As before, we recall the weighted sampling operator $Q_\Omega$, the empirical Gram matrix $\randgram$, and the empirical cross-term $\widehat C_\Omega$.

\begin{theorem}\label{main+frames}
    Let $\mathcal H$ be a separable Hilbert space. Let $\{s_k\}_{k\in\N}$ and $\{w_k\}_{k\in\N}$ be frames for their respective closed spans $\mathcal S$ and $\mathcal W$, with synthesis operators $S$ and $W$. 
    
    Let $f\in\mathcal H$ and let $\mathcal W_n:=\Span\{w_1,\dots,w_n\}$. Suppose we draw $m$ i.i.d. sample indices $\Omega=\{i_1,\dots,i_m\}$ from $\N$ according to a probability distribution $p = (p_j)_{j\in\N}$. Let $y_t:=\<f,s_{i_t}\>$ be the acquired samples. 
    
    Under Assumption \ref{assumptions}, for any $\delta\in(0,1)$, if the sample size $m$ satisfies 
    \[
    m\ge\frac{8}{3} R\|\Sigma\|\|\Sigma^\dagger\|^2\log(2n/\delta),
    \]
    then with probability at least $1-\delta$, the weighted least-squares problem
    \[
    \underset{x\in\mathbb{C}^n}{\operatorname{argmin}}\sum_{t=1}^m\frac{1}{p_{i_t}}|\<W\iota_n x,s_{i_t}\>-y_t|^2
    \]
    admits a unique minimal norm solution $\tilde x$. Moreover, the reconstructed signal $\tilde f:=W\iota_n\tilde x$ satisfies the error bound
    \[
    \|f-\tilde f\|\le\|P_{\mathcal W_n^\perp}f\|\left(1+K_{n,\Omega}^2\right)^\frac{1}{2},
    \]
    where $K_{n,\Omega}:=\|W\iota_n\randgram^\dagger\widehat C_\Omega\|$.
\end{theorem}

Theorem \ref{main+frames} represents a fully stochastic extension of the deterministic frame-based generalized sampling introduced in \cite{adcock2012frames}. While the deterministic framework yields bounds that are difficult to compute in practice, our stochastic result provides explicit, non-asymptotic bounds on the sample complexity required to achieve stable recovery with arbitrary redundant frames.

We remark that several auxiliary results displayed in Subsections \ref{subsec:crossterm} and \ref{subsec:leverage_score} did not specifically rely on the Riesz basis assumption for the reconstruction system. To be precise, in Corollary \ref{half.ONB.case}, we only require the sampling system to be orthonormal (the reconstruction system can remain an arbitrary frame). Furthermore, the leverage-score distribution bounds in Corollary \ref{leverage} directly extend to the frame setting by simply substituting $\Sigma^{-1}$ with $\Sigma^\dagger$. 

Analogously, the connection to Christoffel sampling at the end of Subsection \ref{subsec:leverage_score} extends to the frame scenario by substituting every occurrence of $\Sigma^{-1}$ with $\Sigma^\dagger$. (Note that in equation \eqref{christoffel_riesz}, the identity $\Sigma\Sigma^\dagger v_j=v_j$ remains valid by virtue of the fact that $v_j\in\ran(\iota_n^*U^*)=\ran(\Sigma)$ in conjunction with $\Sigma\Sigma^\dagger=P_{\ran(\Sigma)}$). We thus deduce the bounds:
\[
\frac{1}{\|\Sigma\|}R\le\kappa_w(\mathcal P)\le\|\Sigma^\dagger\|R.
\]

We also state the analogue of Theorem \ref{K_n,omega} regarding the concentration of the cross-term factor $K_{n,\Omega}$ in the frame setting.

\begin{theorem}\label{K_n,omega+frames}
    Consider the setting of Theorem \ref{main+frames}. Let $\Lambda:=1+\|\Sigma^\dagger\|+\|C\|$ and let $\varepsilon\in(0,\Lambda\sqrt D\min\{1,2\|\Sigma^\dagger\|\})$ and $\delta\in(0,1)$. If
    \[
    m\ge M_{\varepsilon,\delta}^*:=\max\left\{\tilde M_{\frac{\varepsilon}{\Lambda\sqrt D},\frac{\delta}{2}},\tilde M_{\frac{\varepsilon}{4\|\Sigma^\dagger\|^2\Lambda\sqrt D},\frac{\delta}{2}}\right\},
    \]
    then
    \[
    \P\left(\left|K_{n,\Omega}-\|W\iota_n\Sigma^\dagger C\|\right|<\varepsilon\right)\ge1-\delta.
    \]
\end{theorem}

\smallskip
\textit{Proof Strategy.}
We now sketch the proofs of both theorems. The core difficulty in the frame setting is that $\Sigma$ is generally singular, meaning we cannot rely on standard matrix inversion. We recover the uniqueness of the reconstruction by computing the minimal norm least-squares solution via the Moore-Penrose pseudo-inverse: $\tilde x:=\randgram^\dagger\iota_n^*U^*Q_\Omega\mathbf y$.

This introduces two major analytical hurdles:
\begin{enumerate}
    \item \textit{Algebraic cancellation.} In the Riesz basis setting, bounding the reconstruction error relies on the clean cancellation $\randgram^{-1}\randgram=I_{\C^n}$. With frames, this product yields an orthogonal projection $\randgram^\dagger\randgram=P_{\ran(\randgram)}$, which severely complicates the algebraic decomposition of the projection error.
    \item \textit{Discontinuity of the pseudo-inverse.} To prove that the factor $K_{n,\Omega}$ concentrates, we require $\randgram^\dagger \xrightarrow{\P} \Sigma^\dagger$. However, the pseudo-inverse mapping $A \mapsto A^\dagger$ is notoriously discontinuous under arbitrary perturbations if the rank of the matrix changes.
\end{enumerate}

The key to overcoming both hurdles lies in establishing the \textit{high-probability eventual-range stability} of the empirical Gram matrix. The pivotal result of this subsection is that, for sufficiently large $m$, the empirical Gram $\randgram$ and its expected value $\Sigma$ share the same range with high probability. 
\begin{prop}\label{rank.stability}
    One has
    \[
    \ran(\Sigma)=\ran(\iota_n^*W^*W\iota_n).
    \]
    Furthermore, for every $\delta\in(0,1)$ and every $m\ge M_\delta$ one has
    \[
    \ran(\randgram)=\ran(\Sigma)=\ran(\iota_n^*W^*W\iota_n)
    \]
    with probability at least $1-\delta$.
\end{prop}
Once the ranges of the matrices coincide, the pseudo-inversion map becomes continuous when restricted to this subspace (Lemma \ref{cont.pseudo-inv}), allowing us to conclude that $\randgram^\dagger\xrightarrow{\P}\Sigma^\dagger$. Furthermore, the matching ranges allow the cumbersome orthogonal projections in the error decomposition to simplify algebraically, enabling the proof to conclude exactly as it did in the Riesz case. The rigorous event-tracking for these probabilities is detailed in Section~\ref{subsec:proofs_frames}.

\section{Application: reconstruction of analytic functions}\label{sec:analytic.func}

In this section, we consider the classical problem of recovering an analytic function defined on a compact interval from a finite number of its Fourier coefficients. It is natural to reconstruct such a function using a polynomial basis, as the orthogonal projection of an analytic function onto the space of complex-coefficient polynomials of degree at most $n$ converges exponentially fast as $n\to+\infty$ \cite{davis1975interpolation}.

Using our notation, let $\mathcal W=\mathcal S=\mathcal H=L^2([-1,1])$, and let $f\in L^2([-1,1])$ be analytic on $[-1,1]$ (meaning it admits an analytic continuation to an open set containing $[-1,1]$). Since our stochastic framework is independent of the specific choice of bases, we can set the sampling system $\{s_l\}_{l\in\N}$ to be the standard Fourier basis:
\[
s_l(x):=\frac{1}{\sqrt 2}e^{\pi i\sigma(l)x},\quad l\in\N,
\]
where $\sigma\colon\N\rightarrow\Z$ is a bijection. For the reconstruction system, we employ the normalized Legendre polynomials:
\[
w_{k+1}(x):=\frac{\sqrt{k+\frac{1}{2}}}{2^k k!}\frac{d^k}{dx^k}((x^2-1)^k),\quad k\in\N.
\]

Our goal is to answer the following: \emph{If we have access to $m$ Fourier coefficients of $f$, how large must $m$ be to guarantee stable recovery of the first $n$ Legendre coefficients with high probability?}

Previous studies \cite{adcock2015generalized,hrycak2010pseudospectral,demanet2019stable} have shown that deterministic sampling strategies for this Fourier-Legendre pair suffer from severe numerical instability, requiring a stringent quadratic oversampling rate of $m \gtrsim  n^2$ to ensure stable recovery. In contrast, Corollary \ref{cor:onb_nlogn} indicates that by sampling frequencies randomly according to the leverage-score distribution, this barrier can be broken, reducing the requirement to $m \gtrsim n \log(2n/\delta)$.

To apply our theorem, we must explicitly compute this optimal leverage-score distribution $p_l = \|v_l\|^2 / n$. By definition, the vector $v_l \in \C^n$ consists of the inner products between the $l$-th sampling vector and the first $n$ reconstruction vectors. Consequently, its squared norm is precisely the sum of the squared Fourier coefficients of the first $n$ Legendre polynomials: 
\[
\|v_l\|^2 = \sum_{k=1}^n |\widehat{w}_k(\sigma(l))|^2,
\]
where, for $l\in\N$, the $2$-periodic $\sigma(l)$-th Fourier coefficient of $f$ is
\[
\widehat f(\sigma(l))=\<f,s_l\>_{L^2([-1,1])}=\frac{1}{\sqrt 2}\int_{-1}^1 f(x)e^{-\pi i\sigma(l)x}\,dx.
\]
Using the well-known closed-form expression for these coefficients in terms of spherical Bessel functions of the first kind, we can explicitly construct the optimal sampling distribution. This yields the following main result for analytic function recovery.

\begin{theorem}\label{analytic}
    Let $m,n\in\N$ and let $f$ be an analytic function on $[-1,1]$. Let $\{s_k\}_{k\in\N}$ be the Fourier basis and $\{w_k\}_{k\in\N}$ the normalized Legendre polynomial basis. Sample $m$ independent, identically distributed indices $\Omega=\{i_1,\dots,i_m\}\subseteq\N$ according to the probability distribution:
    \[
    p_l=\frac{\|v_l\|^2}{n}=\frac{1}{n}\sum_{k=0}^{n-1}(2k+1)\left|j_k(-\sigma(l)\pi)\right|^2,
    \]
    where $j_k$ is thev $k$-th spherical Bessel function of the first kind.
    For $\delta\in(0,1)$, if $m\ge\frac{8}{3}n\log(2n/\delta)$, then with probability at least $1-\delta$, the least-squares problem
    \[
    \underset{x\in\mathbb{C}^n}{\operatorname{argmin}}\sum_{t=1}^m\frac{1}{p_{i_t}}|\widehat{W\iota_n x}(\sigma(i_t))-\widehat f(\sigma(i_t))|^2
    \]
    admits a unique solution $\tilde x$. Moreover, the reconstruction $\tilde f:=W\iota_n\tilde x$ satisfies:
    \[
    \|f-\tilde f\|_{L^2([-1,1])}\lesssim\sqrt n\rho^{-n}\left(1+K_{n,\Omega}^2\right)^\frac{1}{2},
    \]
    for some $\rho>1$, where the cross-term factor $K_{n,\Omega}:=\|W\iota_n\randgram^{-1}\widehat C_\Omega\|$ converges in probability to $0$ as $m\to+\infty$.
\end{theorem}

\begin{remark}[Overcoming the deterministic barrier]
    By employing the optimal stochastic sampling rate $m\asymp n\log(n/\delta)$, we can select a reconstruction dimension $n\asymp\frac{m}{\log(m/\delta)}$. Because the cross-term vanishes ($K_{n,\Omega} \xrightarrow{\P} 0$), the asymptotic bound on the reconstruction error becomes:
    \[
    \|f-\tilde f\|_{L^2([-1,1])}\le C\sqrt m\,\rho^{-\frac{m}{\log(m/\delta)}}.
    \]
    This represents a drastic and fundamental improvement over the deterministic setting. In classical deterministic Generalized Sampling, the severe algebraic mismatch between the Fourier and Legendre bases yields an exponentially ill-conditioned system, forcing the quadratic sample complexity $m \gtrsim n^2$. Under that deterministic restriction, the maximum reconstructable polynomial degree is  $n \sim \sqrt{m}$, which limits the convergence to a \textit{root-exponential} rate of $\mathcal{O}(\rho^{-\sqrt{m}})$. 
    
    In stark contrast, our leverage-score sampling automatically selects the most informative frequencies. By allowing a near-linear reconstruction degree $n \sim m/\log m$, the stochastic framework yields \textit{near-exponential} rate of $\mathcal{O}(\rho^{-m/\log m})$. This rigorously demonstrates that geometry-aware randomization not only stabilizes the inversion process, but also improves the convergence rates of infinite-dimensional inverse problems.
\end{remark}

\section{Proofs}\label{sec:proofs}

\subsection{Concentration Inequalities}\label{subsec:proofs_concentration}

The driving force behind our main results is the matrix Bernstein inequality, which allows us to bound the probability that the norm of a sum of independent, zero-mean random matrices exceeds a given tolerance by the dimension of the matrices times an exponentially decaying factor. This is the key in obtaining the desired sampling rate of the form $m\gtrsim n\log(2n/\delta)$.
\begin{theorem}[Bernstein inequality {\cite[Theorem 1.4]{tropp2012user}}]\label{bernstein}
    Let $X_1,\dots,X_m$ be independent, random, self-adjoint matrices of dimension $d$ such that $\E(X_t)=0$ for every $t\in\{1,\dots,m\}$. Assume that there exists $L\ge0$ such that for every $t\in\{1,\dots,m\}$ one has $\|X_t\|\le L$ almost surely.
    Then for every $\varepsilon>0$ one has
    \[
    \P\left(\left\|\sum_{t=1}^m X_t\right\|\ge\varepsilon\right)\le2d\cdot\exp\left(\frac{-\varepsilon^2/2}{V+L\varepsilon/3}\right),
    \]
    where $V:=\|\sum_{t=1}^m\E[X_t^2]\|$.
\end{theorem}

We need an infinite-dimensional version of the Bernstein inequality since, in Subsection~\ref{subsec:crossterm}, we wish to apply to the empirical cross-term $\widehat C_\Omega$ the same reasoning we used for the empirical Gram $\randgram$. The first issue we stumble upon is that Theorem \ref{bernstein} makes use of the dimension $d$ of the matrices in play. Fortunately, there exists a dimension-free version of the Bernstein inequality that employs the effective rank of a matrix, defined as 
\[
r(A):=\frac{\tr(A)}{\|A\|}.
\]
Since the definition of effective rank can be extended to trace-class operators, this version has been generalized to random operators acting on a Hilbert space.
\begin{theorem}[Bernstein inequality for random, self-adjoint, Hilbert-Schmidt operators {\cite[Section 3.2]{minsker2017some}}]\label{op.bernstein}
    Let $\mathcal H$ be a separable Hilbert space and let $X_1,\dots,X_m$ be independent, random, self-adjoint, Hilbert-Schmidt operators on $\mathcal H$ such that $\E(X_t)=0$ for every $t\in\{1,\dots,m\}$. Assume that there exists $L,\sigma^2\ge0$ such that for every $t\in\{1,\dots,m\}$ one has:
    \begin{align*}
    &\|X_t\|\le L\quad\text{almost surely},\\
    &\left\|\sum_{t=1}^m\E(X_t^2)\right\|\le\sigma^2.
    \end{align*}
    Then for every $\varepsilon\ge\frac{1}{6}(L+\sqrt{L^2+36\sigma^2})$ one has
    \[
    \P\left(\left\|\sum_{t=1}^m X_t\right\|\ge\varepsilon\right)\le 14\cdot r\!\left(\sum_{t=1}^m\E(X_t^2)\right)\exp\left(\frac{-\varepsilon^2/2}{\sigma^2+L\varepsilon/3}\right).
    \]
\end{theorem}
At the cost of a worse constant and of a constraint on $\varepsilon$, this extension of the Bernstein inequality solves the dimensionality issue. However, the cross-terms $\widehat C_\Omega,C$ do not act on a single space, but on two distinct Hilbert spaces:
\[
\widehat C_\Omega,C\colon\mathcal H\longrightarrow\C^n.
\]
Our goal is now to further extend Theorem \ref{op.bernstein} to random operators mapping between possibly distinct Hilbert spaces $\mathcal H_1\longrightarrow\mathcal H_2$. This is done by means of the so-called Hermitian dilation, a technique that is employed to extend result from square self-adjoint matrices to rectangular matrices (see \cite[Theorem 1.6]{tropp2012user}).
\begin{definition}\label{def:Herm}
    Let $T\colon\mathcal H_1\longrightarrow\mathcal H_2$ be a bounded operator between Hilbert spaces. The Hermitian dilation of $T$ is
    \begin{align*}
        H(T)\colon\mathcal H_1\oplus\mathcal H_2&\longrightarrow\mathcal H_1\oplus\mathcal H_2,\\
        (x,y)&\longmapsto(T^*y,Tx).
    \end{align*}
\end{definition}
We have now at our disposal all the ingredients to prove the extension of the Bernstein inequality that we need.
\begin{proof}[Proof of Theorem \ref{op.rect.bernstein}]
    We wish to apply Theorem \ref{op.bernstein} to the dilations $H(X_t)$. We have that $H(X_1),\dots,H(X_m)$ are independent, random, self-adjoint operators on $\mathcal H_1\oplus\mathcal H_2$ (where self-adjointness follows from Proposition \ref{hermitian.dilation} (2)) satisfying $\E(H(X_t))=H(\E(X_t))=0$ for every $t\in\{1,\dots,m\}$, where $H$ commutes with $\E$ by continuity (indeed, $H$ is isometric from Proposition \ref{hermitian.dilation} (1)) and $\R$-linearity (Proposition \ref{hermitian.dilation} (3)). We must check that each $H(X_t)$ is a Hilbert-Schmidt operator: this is equivalent to requiring $H(X_t)^*H(X_t)=H(X_t)^2$ to be trace-class. A simple calculation shows that $H(X_t)^2(x,y)=(X_t^*X_tx,X_tX_t^*y)$ for every $(x,y)\in\mathcal H_1\oplus\mathcal H_2$: it follows that 
    \[
    \tr(H(X_t)^2)=\tr(X_t^*X_t)+\tr(X_tX_t^*),
    \]
    and this is finite by hypothesis. We also deduce that $\sum_{t=1}^m\E(H(X_t)^2)=(V_1,V_2)$, and it's easy to see that
    \begin{align*}
        &\tr((V_1,V_2))=\tr(V_1)+\tr(V_2),\\
        &\|(V_1,V_2)\|=\max\{\|V_1\|,\|V_2\|\}\le\sigma^2.
    \end{align*}
    This allows us to write the effective rank of $\sum_{t=1}^m\E(H(X_t)^2)$ in the following way:
    \begin{align*}
    r\left(\sum_{t=1}^m\E(H(X_t)^2)\right)&=\frac{\tr(V_1)+\tr(V_2)}{\max\{\|V_1\|,\|V_2\|\}}.
    \end{align*}
    Let $\varepsilon\ge\frac{1}{6}(L+\sqrt{L^2+36\sigma^2})$. Since $\|H(X_t)\|=\|X_t\|\le L$ for each $t$, we can apply Theorem \ref{op.bernstein} to infer
    \begin{align*}
        \P\left(\left\|\sum_{t=1}^m H(X_t)\right\|\ge\varepsilon\right)&\le 14\cdot r\!\left(\sum_{t=1}^m\E(H(X_t)^2)\right)\exp\left(\frac{-\varepsilon^2/2}{\sigma^2+L\varepsilon/3}\right)\\
        &=14\frac{\tr(V_1)+\tr(V_2)}{\max\{\|V_1\|,\|V_2\|\}}\exp\left(\frac{-\varepsilon^2/2}{\sigma^2+L\varepsilon/3}\right).
    \end{align*}
    And the claim follows because $\|\sum_{t=1}^m H(X_t)\|=\|H(\sum_{t=1}^m X_t)\|=\|\sum_{t=1}^m X_t\|$, where we used the $\R$-linearity of $H$ (Proposition \ref{hermitian.dilation} (3)).
\end{proof}

\subsection{Proof of the Main Result}

Recall that, for this subsection, we are assuming the sampling system to be a frame and the reconstruction system to be a Riesz basis. Under these assumption, the matrix $\Sigma$ is invertible.
\begin{prop}\label{sigma.invertible}
    $\Sigma=\iota_n^*U^*U\iota_n$ is invertible.
\end{prop}
\begin{proof}
    We will start by showing that $U=S^* W$ is injective: let $x\in \ker(U)$. We can check that $Wx\in\mathcal S^\perp$: let $s\in\mathcal S$. Since $\mathcal S=S(\ell^2)$, let $y\in\ell^2$ such that $s=Sy$. Then
    \[
    \<Wx,s\>=\<Wx,Sy\>=\<S^*Wx,y\>_{\ell^2}=\<Ux,y\>_{\ell^2}=0.
    \]
    Clearly, we also have $Wx\in\mathcal W$. Since we are assuming $\mathcal W\cap\mathcal S^\perp=\{0\}$ (see Assumptions \ref{assumptions}), we conclude that $Wx=0$. But then $x=0$ by injectivity of the synthesis operator $W$. Now, since $\Sigma$ is an $n\times n$ matrix, invertibility is equivalent to injectivity. Furthermore, since $\Sigma=\iota_n^*U^*U\iota_n=(U\iota_n)^*U\iota_n$, Lemma \ref{kerA*A} ensures that $\ker(\Sigma)=\ker(U\iota_n)$, and the latter is $\{0\}$ by injectivity of $U$ and $\iota_n$.
\end{proof}
In order to employ the matrix Bernstein inequality (Theorem \ref{bernstein}), we need to express the difference $\randgram-\Sigma$ as the sum of independent, centered, self-adjoint matrices. This is the essence of the next lemma. Recall that $v_j:=\iota_n^*U^*e_j\in\C^n$.
\begin{lemma}\label{repr.sigma}
    One has:
    \[
    \Sigma=\sum_{j\in\N}v_j\otimes v_j\qquad\text{and}\qquad\randgram=\frac{1}{m}\sum_{t=1}^m\frac{v_{i_t}\otimes v_{i_t}}{p_{i_t}},
    \]
    where the infinite series converges in the operator norm topology. In particular, in accordance to our convention \eqref{supp(p)}, we have
    \[
    \Sigma=\sum_{j\in\supp(p)}v_j\otimes v_j.
    \]
\end{lemma}
\begin{proof}
    Use the properties of rank-1 operators to quickly deduce
    \begin{align*}
        \randgram&=\iota_n^*U^*Q_{\Omega}U\iota_n\\
        &=\frac{1}{m}\sum_{t=1}^m\iota_n^*U^*\left(\frac{e_{i_t}\otimes e_{i_t}}{p_{i_t}}\right)U\iota_n\\
        &=\frac{1}{m}\sum_{t=1}^m\frac{(\iota_n^*U^*e_{i_t})\otimes((U\iota_n)^*e_{i_t})}{p_{i_t}}\\
        &=\frac{1}{m}\sum_{t=1}^m\frac{v_{i_t}\otimes v_{i_t}}{p_{i_t}}.
    \end{align*}
    We can apply the same reasoning to $\Sigma$. However, the series $\sum_{j\in\N}e_{j}\otimes e_j$ converges to the identity only in the strong sense, but not in norm. It follows that the series $\sum_{j\in\N}\iota_n^*U^*(e_j\otimes e_j)U\iota_n=\sum_{j\in\N}v_j\otimes v_j$ converges strongly to $\iota_n^*U^*IU\iota_n=\Sigma$. But strong operator topology and norm operator topology coincide in the space of operators that act on a finite dimension space (i.e. the space of matrices), and the claim is proved.
\end{proof}
We arrived at the heart of the proof of Theorem \ref{mainresult}: we apply the Bernstein inequality and carry on the computations to explicitly derive the stable sampling threshold $M_{\varepsilon,\delta}$.
\begin{prop}\label{prob.conv}
     For every $\varepsilon\in(0,\|\Sigma\|],\delta\in(0,1)$ there exists $M_{\varepsilon,\delta}>0$ such that for every natural $m\ge M_{\varepsilon,\delta}$ one has
     \[
     \P(\|\randgram-\Sigma\|\ge\varepsilon)\le\delta.
     \]
\end{prop}
\begin{proof}
    Fix $\varepsilon\in(0,\|\Sigma\|],\delta\in(0,1)$. The strategy of the proof revolves around bounding the tail $\P(\|\randgram-\Sigma\|\ge\varepsilon)$ using the matrix Bernstein inequality (Theorem \ref{bernstein}). To this end, we must write $\randgram-\Sigma$ as the sum of independent, random, self-adjoint matrices, and this is where Lemma \ref{repr.sigma} comes into play: define for every $t\in\{1,\dots,m\}$
    \begin{align*}
        Y_t:=\frac{v_{i_t}\otimes v_{i_t}}{p_{i_t}},\qquad X_t:=\frac{1}{m}(Y_t-\Sigma).
    \end{align*}
    Clearly, we have $\randgram-\Sigma=\sum_{t=1}^mX_t$. Each $Y_t$ is self-adjoint and so is $\Sigma$: it follows that $X_t$ is self-adjoint for every $t\in\{1,\dots,m\}$. Furthermore, since the random indices $i_1,\dots, i_m$ are independent, we have that $Y_1,\dots,Y_m$ are independent random matrices, and the same applies to $X_1,\dots,X_m$. The expected value of each $Y_t$ is
    \[
    \E(Y_t)=\sum_{j\in\supp(p)}p_j\frac{v_j\otimes v_j}{p_j}=\sum_{j\in\supp(p)}v_j\otimes v_j=\Sigma,
    \]
    (where the last step follows from Lemma \ref{repr.sigma}), and this entails $\E(X_t)=0$ for every $t$.

    Next, let $t\in\{1,\dots,m\}$. We can use the fact that $X_t$ is self-adjoint to evaluate its operator norm as follows:
    \[
    \|X_t\|=\sup_{\substack{x\in\mathbb{C}^n \\ \|x\|=1}}|\<X_tx,x\>|.
    \]
    For every $x\in\C^n$ such that $\|x\|=1$ we have
    \begin{align*}
        |\<X_tx,x\>|&=\frac{1}{m}\left|\frac{1}{p_{i_t}}\<(v_{i_t}\otimes v_{i_t})x,x\>-\<\Sigma x,x\>\right|\\
        &=\frac{1}{m}\left|\frac{1}{p_{i_t}}\<\<x,v_{i_t}\>v_{i_t},x\>-\<\Sigma x,x\>\right|\\
        &=\frac{1}{m}\left|\frac{1}{p_{i_t}}\<x,v_{i_t}\>\<v_{i_t},x\>-\<\Sigma x,x\>\right|\\
        &=\frac{1}{m}\left|\frac{1}{p_{i_t}}|\<x,v_{i_t}\>|^2-\<\Sigma x,x\>\right|\\
        &\le\frac{1}{m}\max\left\{\frac{1}{p_{i_t}}|\<x,v_{i_t}\>|^2,\<\Sigma x,x\>\right\}
    \end{align*}
    Notice that $\frac{1}{p_{i_t}}|\<x,v_{i_t}\>|^2\le\frac{1}{p_{i_t}}\|x\|^2\|v_{i_t}\|^2\le  R$, where the last inequality follows from the definition of $R$. We also have that $\<\Sigma x,x\>\le\|\Sigma\|\le R$: this gives us the bound
    \begin{align}\label{bound.L}
        \|X_t\|\le\frac{R}{m}=:L.
    \end{align}
    Next, we will find a suitable bound for the variance parameter $V:=\|\sum_{t=1}^m\E(X_t^2)\|$. Start with
    \begin{align*}
        &X_t^2=\frac{1}{m^2}(Y_t-\Sigma)^2=\frac{1}{m^2}(Y_t^2-2Y_t\Sigma+\Sigma^2),\\
        &\E(X_t^2)=\frac{1}{m^2}(\E(Y_t^2)-2\E(Y_t)\Sigma+\Sigma^2)=\frac{1}{m^2}(\E(Y_t^2)-2\Sigma^2+\Sigma^2)=\frac{1}{m^2}(\E(Y_t^2)-\Sigma^2).
    \end{align*}
    For every $x\in\C^n$ with $\|x\|=1$ one has
    \begin{align*}
        \<\E(X_t^2)x,x\>=\frac{1}{m^2}(\<\E(Y_t^2)x,x\>-\<\Sigma^2x,x\>).
    \end{align*}
    Use the fact that $Y_t^2=\frac{\|v_{i_t}\|^2v_{i_t}\otimes v_{i_t}}{p_{i_t^2}}=\frac{\|v_{i_t}\|^2}{p_{i_t}}Y_t$ to bound the first summand in the following way:
    \begin{align*}
        \<\E(Y_t^2)x,x\>&=\E\left(\frac{\|v_{i_t}\|^2}{p_{i_t}}\<Y_tx,x\>\right)\\
        &\le R\E(\<Y_tx,x\>)\\
        &= R\<\E(Y_t)x,x\>\\
        &= R\<\Sigma x,x\>\\
        &\le R\|\Sigma\|.
    \end{align*}
    On the other hand, we have
    \[
    \<\Sigma^2x,x\>=\|\Sigma x\|^2\ge0.
    \]
    Ultimately, we are left with
    \[
    \<\E(X_t^2)x,x\>\le\frac{ R\|\Sigma\|}{m^2}.
    \]
    Since the right-hand side is independent of $t$,
    \[
    \left\<\sum_{t=1}^m\E(X_t^2)x,x\right\>\le\frac{ R\|\Sigma\|}{m}.
    \]
    Taking the supremum as $\|x\|=1$, we deduce 
    \begin{align}\label{bound.V}
        V\le\frac{ R\|\Sigma\|}{m}.
    \end{align}
    We can now use the matrix Bernstein inequality (Theorem \ref{bernstein}) to infer
    \begin{align*}
        \P\left(\|\randgram-\Sigma\|\ge\varepsilon\right)\le 2n\cdot\exp\left(\frac{-\varepsilon^2/2}{V+L\varepsilon/3}\right).
    \end{align*}
    Substituting the value of $L$ \eqref{bound.L} and using the bound for $V$ \eqref{bound.V}, we get
    \[
    \P\left(\|\randgram-\Sigma\|\ge\varepsilon\right)\le 2n\cdot\exp\left(\frac{-m\varepsilon^2/2}{ R\|\Sigma\|+ R\varepsilon/3}\right).
    \]
    The claim follows if, for sufficiently large $m$, we can prove that
    \[
    2n\cdot\exp\left(\frac{-m\varepsilon^2/2}{ R\|\Sigma\|+ R\varepsilon/3}\right)\le\delta,
    \]
    which is the case if
    \begin{align}\label{m.lowerbound}
        m&\ge2 R\left(\|\Sigma\|+\frac{\varepsilon}{3}\right)\frac{\log(2n/\delta)}{\varepsilon^2}.
    \end{align}
    Since $\varepsilon\le\|\Sigma\|$, \eqref{m.lowerbound} is true if, for example,
    \begin{align*}
        m&\ge2 R\left(\|\Sigma\|+\frac{\|\Sigma\|}{3}\right)\frac{\log(2n/\delta)}{\varepsilon^2}\\
        &=\frac{8}{3} R\|\Sigma\|\frac{\log(2n/\delta)}{\varepsilon^2}=:M_{\varepsilon,\delta}. \qedhere
    \end{align*}
\end{proof}
If $\randgram$ concentrates around $\Sigma$, which is invertible by Lemma \ref{sigma.invertible}, then $\randgram$ itself must be invertible with high probability for sufficiently large $m$.
\begin{corollary}\label{randgram.inv}
    For every $m\ge M_\delta:=M_{\lambda_0,\delta}=\frac{8}{3} R\|\Sigma\|\lambda_0^{-2}\log(2n/\delta)$ one has that $\randgram$ is invertible with probability at least $1-\delta$.
\end{corollary}
\begin{proof}
     By Lemma \ref{norm.of.inverse},
     \[
     \left\{\|\randgram-\Sigma\|<\lambda_0=\|\Sigma^{-1}\|^{-1}\right\}\subseteq\{\randgram\text{ is invertible}\}.
     \]
     It thus suffices to apply Proposition \ref{prob.conv} with $\varepsilon=\lambda_0\le\|\Sigma\|$ to conclude that for every $m\ge M_{\lambda_0,\delta}=M_\delta$
   \begin{equation*}
          \P\left(\randgram\text{ is invertible}\right)\ge\P\left(\|\randgram-\Sigma\|<\lambda_0=\|\Sigma^{-1}\|^{-1}\right)\ge1-\delta. \qedhere
   \end{equation*}
\end{proof}
To conclude the proof of Theorem \ref{mainresult}, we are going to estimate the reconstruction error $\|f-\tilde f\|$. 

\begin{proof}[Proof of Theorem \ref{mainresult}]
    Let $m\ge M_\delta$: by Corollary \ref{randgram.inv}, we can define
    \[
    \tilde f:=W\iota_n\randgram^{-1}\iota_n^*U^*Q_{\Omega}S^*f
    \]
    on a set of probability at least $1-\delta$. Decompose $f=P_{\mathcal W_n}f+P_{\mathcal W_n^\perp}f$:
    \begin{align*}
        \tilde f=W\iota_n\randgram^{-1}\iota_n^*U^*Q_\Omega S^*P_{\mathcal W_n}f+W\iota_n\randgram^{-1}\iota_n^*U^*Q_\Omega S^*P_{\mathcal W_n^\perp}f=\tilde f_1+W\iota_n\randgram^{-1}\widehat C_\Omega f,
    \end{align*}
    where $\tilde f_1:=W\iota_n\randgram^{-1}\iota_n^*U^*Q_\Omega S^*P_{\mathcal W_n}f$. Use the fact that $$P_{\mathcal W_n}=P_{\ran(W\iota_n)}=W\iota_n(\iota_n^*W^*W\iota_n)^{-1}\iota_n^*W^*$$ to write
    \begin{align*}
        \tilde f_1=W\iota_n\randgram^{-1}\iota_n^*U^*Q_\Omega S^*W\iota_n(\iota_n^*W^*W\iota_n)^{-1}\iota_n^*W^*f.
    \end{align*}
    Next, recognize $\iota_n^*U^*Q_\Omega S^*W\iota_n=\randgram$ to derive
    \[
    \tilde f_1=W\iota_n(\iota_n^*W^*W\iota_n)^{-1}\iota_n^*W^*f=P_{\mathcal W_n}f,
    \]
    hence
    \[
    \tilde f=P_{\mathcal W_n}f+W\iota_n\randgram^{-1}\widehat C_\Omega f.
    \]
    It follows that
    \[
    f-\tilde f=P_{\mathcal W_n^\perp}f-W\iota_n\randgram^{-1}\widehat C_\Omega f.
    \]
    Finally, by the Pythagorean theorem,
    \begin{align*}
        \|f-\tilde f\|&=(\|P_{\mathcal W_n^\perp}f\|^2+\|W\iota_n\randgram^{-1}\widehat C_\Omega f\|^2)^{\frac{1}{2}}\le\|P_{\mathcal W_n^\perp}f\|(1+K_{n,\Omega}^2)^{\frac{1}{2}}.\qedhere
    \end{align*}
\end{proof}
Under the additional geometric assumption $\mathcal W\subseteq\mathcal S$, one is able to uniformly bound the condition number of $\Sigma$.
\begin{prop}\label{unifbound.cond_number}
If $\mathcal W\subseteq\mathcal S$, then $\kappa(\Sigma):=\|\Sigma\|\|\Sigma^{-1}\|\le\frac{BD}{AC}$.
\end{prop}
\begin{proof}
    First, we have $\|\Sigma\|=\|\iota_n^*U^*U\iota_n\|\le\|U\|^2\le\|S\|^2\|\|W\|^2\le BD$. Since we know that $\|S\alpha\|\ge\sqrt A\|\alpha\|$ for every $\alpha\in\ell^2$, Lemma \ref{below.bound.adjoint} entails that $\|S^*S\alpha\|\ge\sqrt A\|S\alpha\|$ for every $\alpha\in\ell^2$, or $\|S^*s\|\ge\sqrt A\|s\|$ for every $s\in\mathcal S$. From the assumption $\mathcal W\subseteq \mathcal S$, we deduce that $\|U\iota_n x\|=\|S^*W\iota_n x\|\ge\sqrt A\|W\iota_n x\|\ge\sqrt{AC}\|x\|$ for every $x\in\C^n$. But then we have
    \[
    \|\Sigma x\|=\|(U\iota_n)^*U\iota_n x\|\ge\sqrt{AC}\|U\iota_n x\|\ge AC\|x\|\quad\text{for every }x\in\C^n,
    \]
    again by Lemma \ref{below.bound.adjoint}. This allows us to conclude that $\|\Sigma^{-1}\|\le(AC)^{-1}$, hence
    \begin{equation*}
    \kappa(\Sigma)=\|\Sigma\|\|\Sigma^{-1}\|\le\frac{BD}{AC}. \qedhere
    \end{equation*}
\end{proof}
This immediately yields Corollary\ref{cor:Rlogn}.
\begin{proof}[Proof of Corollary \ref{cor:Rlogn}]
    Just like in the proof of Proposition \ref{unifbound.cond_number}, we have that $\|\Sigma\|\|\Sigma^{-1}\|^2\le\frac{BD}{A^2C^2}$. It follows that the conclusions of Theorem \ref{mainresult} hold provided that
    \[
    m\ge\frac{8}{3}\frac{BD}{A^2C^2}R\log(2n/\delta). \qedhere
    \]
\end{proof}

\subsection{Proofs for the Empirical Cross-term}

Analogously to the Gram case, we need to express the difference $\widehat C_\Omega-C$ as the sum of independent, centered operators. We recall that $u_j:=P_{\mathcal W_n^\perp}s_j\in\mathcal W_n^\perp$.
\begin{lemma}\label{repr.C}
    One has:
    \[
    C=\sum_{j\in\N}v_j\otimes u_j\qquad\text{and}\qquad\widehat C_\Omega=\frac{1}{m}\sum_{t=1}^m\frac{v_{i_t}\otimes u_{i_t}}{p_{i_t}},
    \]
    where the series converges in the operator norm topology. In particular, in accordance to our convention \eqref{supp(p)}, we have
    \[
    C=\sum_{j\in\supp(p)}v_j\otimes u_j.
    \]
\end{lemma}
\begin{proof}
    The representation of $\widehat C_\Omega$ immediately follows from the definition of $Q_\Omega$ and the properties of rank-1 operators:
    \begin{align*}
        \widehat C_\Omega&=\iota_n^*U^*Q_\Omega S^*P_{\mathcal W_n^\perp}\\
        &=\frac{1}{m}\sum_{t=1}^m\iota_n^*U^*\left(\frac{e_{i_t}\otimes e_{i_t}}{p_{i_t}}\right)S^*P_{\mathcal W_n^\perp}\\
        &=\frac{1}{m}\sum_{t=1}^m\frac{(\iota_n^*U^*e_{i_t})\otimes(P_{\mathcal W_n^\perp}Se_{i_t})}{p_{i_t}}\\
        &=\frac{1}{m}\sum_{t=1}^m\frac{v_{i_t}\otimes u_{i_t}}{p_{i_t}}.
    \end{align*}
    Similarly, it's easy to see that the series $\sum_{j\in\N}v_j\otimes u_j$ converges strongly to $C$. But the series is absolutely convergent:
    \[
    \sum_{j\in\N}\|v_j\otimes u_k\|=\sum_{j\in\N}\|v_j\|\|u_j\|\le \sqrt{RR'}\sum_{j\in\N}p_j=\sqrt{RR'},
    \]
    meaning that it must converge in the operator topology to its pointwise limit $C$.
\end{proof}
We can now apply our extension of the Bernstein inequality.
\begin{prop}\label{crossterm}
    For every $\varepsilon\in(0,\|\Sigma\|],\delta\in(0,1)$ and for every natural 
    \[
    m\ge\tilde M_{\varepsilon,\delta}:=\frac{10}{3}KR''\frac{\log(28n/\delta)}{\varepsilon^2},
    \]
    one has
    \[
    \P(\|\widehat C_\Omega-C\|\ge\varepsilon)\le\delta.
    \]
\end{prop}
\begin{proof}
    Define for every $t\in\{1,\dots,m\}$
    \[
    Y_t:=\frac{v_{i_t}\otimes u_{i_t}}{p_{i_t}},\qquad X_t:=\frac{1}{m}(Y_t-C).
    \]
    By Lemma \ref{repr.C}, we have $\widehat C_{\Omega}-C=\sum_{t=1}^m X_t$. Moreover, $X_1,\dots,X_m$ are independent and their expected value is
    \[
    \E(X_t)=\frac{1}{m}(\E(Y_t)-C),
    \]
    but since
    \[
    \E(Y_t)=\sum_{j\in\supp(p)}p_j\frac{v_j\otimes u_j}{p_j}=\sum_{j\in\supp(p)}v_j\otimes u_j=C,
    \]
    (where the last equality follows from Lemma \ref{repr.C}), we have $\E(X_t)=0$ for each $t$. Notice that $X_t^*X_t$ is trace class because $X_tX_t^*$ is a $n\times n$ matrix. Furthermore,
    \begin{align*}
        \|X_t\|=\frac{1}{m}\|Y_t-C\|\le\frac{1}{m}(\|Y_t\|+\|C\|),
    \end{align*}
    but we have $\|Y_t\|=\frac{\|v_{i_t}\|\|u_{i_t}\|}{p_{i_t}}\le\frac{\sqrt{Rp_{i_t}{R'p_{i_t}}}}{p_{i_t}}=\sqrt{RR'}$ and also
    \begin{equation}\label{convention.in.C}
        \|C\|\le\sum_{j\in\N}\|v_j\|\|u_j\|=\sum_{j\in\supp(p)}\|v_j\|\|u_j\|\le\sqrt{RR'}\sum_{j\in\supp(p)}p_j=\sqrt{RR'},
    \end{equation}
    and this allows us to deduce that
    \[
    \|X_t\|\le\frac{2\sqrt{RR'}}{m}\le\frac{2R''}{m}=:L.
    \]
    We stress that in \eqref{convention.in.C} we were able to sum over the support of the distribution $p$ thanks to our convention \eqref{supp(p)}. While this might seem like a mere technicality, it is not guaranteed that $u_j=0$ when $j\notin\supp(p)$.\\
    Next, define $V_1=\sum_{t=1}^m\E(X_t^*X_t)$ and notice that
    \begin{equation}\label{to.be.expected}
        X_t^*X_t=\frac{1}{m^2}(Y_t^*Y_t-Y_t^*C-C^*Y_t+C^*C).
    \end{equation}
    Now,
    \[
    Y_t^*Y_t=\frac{\|v_{i_t}\|^2}{p_{i_t}^2}u_{i_t}\otimes u_{i_t}=\frac{\|v_{i_t}\|^2}{p_{i_t}}Z_t,
    \]
    where $Z_t:=\frac{u_{i_t}\otimes u_{i_t}}{p_{i_t}}$. On taking the expected value to both sides of \eqref{to.be.expected}, we get
    \[
    \E(X_t^*X_t)=\frac{1}{m^2}\left(\E\left(\frac{\|v_{i_t}\|^2}{p_{i_t}}Z_t\right)-C^*C\right).
    \]
    Moreover
    \[
    \E(Z_t)=\sum_{j\in\supp(p)}p_j\frac{u_j\otimes u_j}{p_j}=\sum_{j\in\supp(p)}u_j\otimes u_j=T.
    \]
    Fix $h\in\mathcal H$ such that $\|x\|=1$. Then
    \begin{align*}
        \<\E(X_t^*X_t)h,h\>&=\frac{1}{m^2}\left(\E\left(\frac{\|v_{i_t}\|^2}{p_{i_t}}\<Z_th,h\>\right)-\<C^*Ch,h\>\right)\\
        &\le\frac{1}{m^2}(R\cdot\E(\<Z_th,h\>)-\|Ch\|^2)\\
        &\le\frac{1}{m^2}(R\<\E(Z_t)h,h\>)\\
        &=\frac{1}{m^2}R\<Th,h\>\\
        &\le\frac{R\|T\|}{m^2}.
    \end{align*}
    Since $V_1$ is a self-adjoint, positive operator on $\mathcal H$, we can estimate its operator norm by
    \begin{equation}\label{bound.V1}
        \|V_1\|=\sup_{\substack{h\in\mathcal H \\ \|h\|=1}}\<V_1h,h\>=\sup_{\substack{h\in\mathcal H \\ \|h\|=1}}\sum_{t=1}^m\<\E(X_t^*X_t)h,h\>\le\frac{R\|T\|}{m}.
    \end{equation}
    An analogous bound holds for the norm of $V_2=\sum_{t=1}^m\E(X_tX_t^*)$: indeed,
    \[
    \E(X_tX_t^*)=\frac{1}{m^2}(\E(Y_tY_t^*)-CC^*),
    \]
    and since $Y_tY_t^*=\frac{\|u_{i_t}\|^2}{p_{i_t}^2}v_{i_t}\otimes v_{i_t}$, for every $x\in\C^n$ such that $\|x\|=1$ we have
    \begin{align*}
        \<\E(X_tX_t^*)x,x\>&=\frac{1}{m^2}(\<\E(Y_tY_t^*)x,x\>-\<CC^*x,x\>)\\
        &=\frac{1}{m^2}\left(\E\left(\frac{\|u_{i_t}\|^2}{p_{i_t}}\left\<\frac{v_{i_t}\otimes v_{i_t}}{p_{i_t}}x,x\right\>\right)-\|C^*x\|^2\right)\\
        &\le\frac{1}{m^2}R'\left\<\E\left(\frac{v_{i_t}\otimes v_{i_t}}{p_{i_t}}\right)x,x\right\>\\
        &=\frac{1}{m^2}R'\<\Sigma x,x\>\\
        &\le\frac{R'\|\Sigma\|}{m^2},
    \end{align*}
    and we obtain
    \begin{equation}\label{bound.V2}
        \|V_2\|\le\frac{R'\|\Sigma\|}{m}.
    \end{equation}
    From the bounds \eqref{bound.V1} and \eqref{bound.V2} we get
    \[
    \max\{\|V_1\|,\|V_2\|\}\le\frac{\max\{R\|T\|,R'\|\Sigma\|\}}{m}\le\frac{R''K}{m}=\frac{LK}{2}=:\sigma^2,
    \]
    where $K:=\max\{\|\Sigma\|,\|T\|\}$. We can apply Theorem \ref{op.rect.bernstein} to infer the following: for every $\varepsilon\ge\frac{1}{6}(L+\sqrt{L^2+36\sigma^2})$ one has
    \begin{align*}
        \P\left(\|\widehat C_\Omega-C\|\ge\varepsilon\right)&\le 14\frac{\tr(V_1)+\tr(V_2)}{\max\{\|V_1\|,\|V_2\|\}}\exp\left(\frac{-\varepsilon^2/2}{\sigma^2+L\varepsilon/3}\right)\\
        &\le14\frac{2\tr(V_2)}{\|V_2\|}\exp\left(\frac{-\varepsilon^2/2}{\sigma^2+L\varepsilon/3}\right)\\
        &\le28n\exp\left(\frac{-\varepsilon^2/2}{\sigma^2+L\varepsilon/3}\right),
    \end{align*}
    where we've employed the cyclic property of the trace to entail
    \[
    \tr(V_1)=\sum_{t=1}^m\E(\tr(X_t^*X_t))=\sum_{t=1}^m\E(\tr(X_tX_t^*))=\tr(V_2),
    \]
    and we also bounded $\frac{1}{\max\{\|V_1\|,\|V_2\|\}}\le\frac{1}{\|V_2\|}$ in order to get the effective rank $r(V_2)=\frac{\tr(V_2)}{\|V_2\|}$ of $V_2$. Finally, since $V_2$ is a $n\times n$ self-adjoint matrix, one has $r(V_2)\le n$ (indeed, under these assumptions, matrix norm = spectral radius $\ge\frac{\text{trace}}{n}$). Now, the inequality $\varepsilon\ge\frac{1}{6}(L+\sqrt{L^2+36\sigma^2})$ amounts to saying
    \[
    L\le\frac{6\varepsilon^2}{3K+2\varepsilon},
    \]
    which in turn implies $m\ge\frac{R''}{3\varepsilon^2}(3K+2\varepsilon)$. On the other hand, simplifying
    \[
    28n\cdot\exp\left(\frac{-\varepsilon^2/2}{\sigma^2+L\varepsilon/3}\right)\le\delta
    \]
    yields
    \begin{equation}\label{final.expression.m}
        m\ge\frac{2}{3}R''(3K+2\varepsilon)\frac{\log(28n/\delta)}{\varepsilon^2}.
    \end{equation}
    Since $28n/\delta>e$, \eqref{final.expression.m} is the most restrictive condition. Because $\varepsilon<\|\Sigma\|\le K$, the claim follows with
    \[
    \tilde M_{\varepsilon,\delta}:=\frac{10}{3}KR''\frac{\log(28n/\delta)}{\varepsilon^2}. \qedhere
    \]
\end{proof}
By combining the concentration inequalities for the random Gram matrix (Proposition \ref{prob.conv}) and for the empirical cross-term operator (Proposition \ref{crossterm}), we are able to provide an explicit concentration estimate for the factor $K_{n,\Omega}$.
\begin{proof}[Proof of Theorem \ref{K_n,omega}]
    Consider the two following events:
    \begin{align*}
        &E_1:=\left\{\|\randgram^{-1}-\Sigma^{-1}\|<\frac{\varepsilon}{\Lambda\sqrt D}\right\},\\
        &E_2:=\left\{\|\widehat C_\Omega-C\|<\frac{\varepsilon}{\Lambda\sqrt D}\right\},
    \end{align*}
    and let $m\ge M^*_{\varepsilon,\delta}:=\max\{\tilde M_{\frac{\varepsilon}{\Lambda\sqrt D},\frac{\delta}{2}},\tilde M_{\frac{\varepsilon}{4\|\Sigma^{-1}\|^2\Lambda\sqrt D},\frac{\delta}{2}}\}$. By assumption, $\frac{\varepsilon}{\Lambda\sqrt D}=\min\{1,\frac{\varepsilon}{\Lambda\sqrt D}\}$. Lemma \ref{cont.comp} infers that
    \[
    E_1\cap E_2\subseteq\{\|\randgram^{-1}\widehat C_\Omega-\Sigma^{-1} C\|<\varepsilon/\sqrt D\}.
    \]
    Moreover, since we have
    \[
    |K_{n,\Omega}-\|W\iota_n\Sigma^{-1} C\||\le\|W\iota_n\randgram^{-1}\widehat C_\Omega-W\iota_n\Sigma^{-1} C\|\le\sqrt D\|\randgram^{-1}\widehat C_\Omega- \Sigma^{-1} C\|
    \]
    It follows that 
    $E_1\cap E_2\subseteq\{|K_{n,\Omega}-\|W\iota_n\Sigma^{-1} C\||<\varepsilon\}$. Therefore, the claim follows if we prove that $\P(E_1\cap E_2)\ge1-\delta$, and we do this by showing that each $E_j$ ($j=1,2$) occurs with probability that exceeds $1-\frac{\delta}{2}$.\\
    Since $m\ge\tilde M_{\frac{\varepsilon}{\Lambda\sqrt D},\frac{\delta}{2}}$, we immediately get $P(E_2)\ge1-\frac{\delta}{2}$ from Proposition \ref{crossterm}. Furthermore, we have by Lemma \ref{cont.inv} that
    \[
    E_1\supseteq\left\{\|\randgram-\Sigma\|<\frac{\varepsilon}{4\|\Sigma^{-1}\|^2\Lambda\sqrt D}\right\},
    \]
    where we also used the fact that $\frac{\varepsilon}{4\|\Sigma^{-1}\|^2\Lambda\sqrt D}=\min\{\frac{1}{2\|\Sigma^{-1}\|},\frac{\varepsilon}{4\|\Sigma^{-1}\|^2\Lambda\sqrt D}\}$ because we are assuming $\varepsilon<2\|\Sigma^{-1}\|\Lambda\sqrt D$. Finally, we have
    \[
    \P(E_1)\ge\P\left(\|\randgram-\Sigma\|<\frac{\varepsilon}{4\|\Sigma^{-1}\|^2\Lambda\sqrt D}\right)\ge1-\frac{\delta}{2}
    \]
    thanks to Proposition \ref{prob.conv} because $m\ge\tilde M_{\frac{\varepsilon}{4\|\Sigma^{-1}\|^2\Lambda\sqrt D},\frac{\delta}{2}}$.
\end{proof}
The fact that $K_{n,\Omega}$ converges to zero in probability as $m\to+\infty$ establishes the quasi-optimality of our method. This can be achieved by requiring the subspace condition $\mathcal W\subseteq\mathcal S$ and the sampling system to be orthonormal. 
\begin{proof}[Proof of Corollary \ref{half.ONB.case}]
     Under the orthonormality assumption of $\{s_k\}_{k\in\N}$, the corresponding synthesis operator $S$ becomes an isometry. It follows that $C=\iota_n^*U^*S^*P_{\mathcal W_n^\perp}=\iota_n^*W^*SS^*P_{\mathcal W_n^\perp}=\iota_n^*W^*P_{\mathcal S}P_{\mathcal W_n^\perp}$. Observe that $\ker(\iota_n^*W^*)=\ran(W\iota_n)^\perp=\mathcal W_n^\perp$, hence $\iota_n^*W^*=\iota_n^*W^*P_{\mathcal W_n}$. But then $C=\iota_n^*W^*P_{\mathcal S}P_{\mathcal W_n^\perp}=\iota_n^*W^*P_{\mathcal W_n}P_{\mathcal S}P_{\mathcal W_n^\perp}$. Finally, the assumption $\mathcal W\subseteq\mathcal S$ forces $\mathcal S^\perp\subseteq\mathcal W^\perp\subseteq\mathcal W_n^\perp$ and consequently $P_{\mathcal W_n}P_{\mathcal S}=P_{\mathcal W_n}-P_{\mathcal W_n}P_{\mathcal S^\perp}=P_{\mathcal W_n}$, and this allows us to conclude that $C=\iota_n^*W^*P_{\mathcal W_n}P_{\mathcal S}P_{\mathcal W_n^\perp}=\iota_n^*W^*P_{\mathcal W_n}P_{\mathcal W_n^\perp}=0$.
\end{proof}

\begin{proof}[Proof of Corollary \ref{cor:leverage+subspace.condition}]
    By Proposition \ref{unifbound.cond_number}, we have $\kappa(\Sigma)\le\frac{BD}{AC}$. The claim thus follows from Corollary \ref{leverage} and the subsequent remark.
\end{proof}

\subsection{Proofs for the Leverage-score Sampling}\label{subsec:proofs_leverage}

All the previous results in this section are valid for an arbitrary probability distribution $(p_j)_{j\in\N}$ that satisfies Assumption \ref{assumptions}: we are now going to implement our findings to a specific sampling strategy. Looking at the definition of $R$, one might think to use $p_j\propto\|v_j\|^2$ for each $j\in\N$. To this end, we need to compute the normalization constant, which turns out to be the trace of $\Sigma$.
\begin{prop}\label{trace}
    One has:
    \[
    \tr(\Sigma)=\sum_{j\in\N}\|v_j\|^2.
    \]
\end{prop}
\begin{proof}
    By the cyclic property of the trace, the trace of $\Sigma$ can be rewritten as
    \[
    \tr(\Sigma)=\tr(\iota_n^*U^*U\iota_n)=\tr(U\iota_n\iota_n^*U^*),
    \]
    which in turn is equal to
    \begin{equation*}
\tr(U\iota_n\iota_n^*U^*)=\sum_{j\in\N}\<U\iota_n\iota_n^*U^*e_j,e_j\> =\sum_{j\in\N}\<\iota_n^*U^*e_j,\iota_n^*U^*e_j\>=\sum_{j\in\N}\<v_j,v_j\> =\sum_{j\in\N}\|v_j\|^2. \qedhere
    \end{equation*}
\end{proof}
It's not a coincidence that leverage-score sampling is the optimal sampling strategy, given the fact that (in the orthonormal setting) it coincides with the so-called Christoffel sampling strategy. The following Proposition allows us to see why this is the case.
\begin{prop}\label{christoffel.onb}
    Let $(\mathcal D,\rho)$ be a measure space and let $\mathcal P\subseteq L^2(\mathcal D,\rho)$ be an $n$-dimensional subspace where pointwise evaluation is defined. Define $\mathcal K_\mathcal P\colon\mathcal D\longrightarrow[0,+\infty)$ by letting $\mathcal K_\mathcal P(x):=\|ev_x\|_{\mathcal P^*}^2$, where $ev_x\colon p\in\mathcal P\mapsto p(x)\in\C$. If $\{\phi_1,\dots,\phi_n\}$ is an orthonormal basis for $\mathcal P$, then for every $x\in\mathcal D$
    \[
    \mathcal K_\mathcal P(x)=\sum_{k=1}^n|\phi_k(x)|^2.
    \]
\end{prop}
\begin{proof}
    Since $\mathcal P$ endowed with the scalar product of $L^2(\mathcal D,\rho)$ is a Hilbert space (where completeness descends from finite dimensionality), there exists a unique $q_x\in\mathcal P$ such that for every $p\in\mathcal P$ one has $p(x)=\<p,q_x\>_{L^2(\mathcal D,\rho)}$. Moreover, one has $\mathcal K_{\mathcal P}(x)=\|ev_x\|_{\mathcal P^*}^2=\|q_x\|_{L^2(\mathcal D,\rho)}^2$. We can easily check that $q_x=\sum_{k=1}^m\overline{\phi_k(x)}\phi_k\in\mathcal P$ is the element we are looking for: indeed, for every $p\in\mathcal P$
    \[
    \<p,q_x\>_{L^2(\mathcal D,\rho)}=\sum_{k=1}^m\phi_k(x)\<p,\phi_k\>_{L^2(\mathcal D,\rho)}=\left(\sum_{k=1}^m\phi_k\<p,\phi_k\>_{L^2(\mathcal D,\rho)}\right)(x)=p(x),
    \]
    meaning that
    \[
    \mathcal K_{\mathcal P}(x)=\|q_x\|_{L^2(\mathcal D,\rho)}^2=\<q_x,q_x\>_{L^2(\mathcal D,\rho)}=q_x(x)=\sum_{k=1}^m\overline{\phi_k(x)}\phi_k(x)=\sum_{k=1}^m|\phi_k(x)|^2. \qedhere
    \]
\end{proof}
The above result allows us to find a closed form for the Christoffel function in the orthonormal framework. Indeed, we deduce that $\mathcal K_\mathcal P(j)=\|v_j\|^2$ for each $j\in\N$.
\begin{proof}[Proof of Lemma \ref{christoffel.function}]
    Let $\{\mathbf{e_1},\dots,\mathbf{e_n}\}$ and $\{e_j\}_{j\in \N}$ be the canonical bases of $\C^n$ and $\ell^2$, respectively. Since $U\iota_n$ is an isometry, we have that $\{Ue_1,\dots,Ue_n\}$ is an orthonormal basis for $\mathcal P$. Thanks to Proposition \ref{christoffel.onb}, for every $j\in\N$ we have
    \begin{align*}
        \mathcal K_\mathcal P(j)=\sum_{k=1}^n|\<Ue_k,e_j\>|^2=\sum_{k=1}^n|\<U^*e_j,e_k\>|^2.
    \end{align*}
    Notice that $\iota_n\mathbf{e_k}=e_k$ whenever $k\in\{1,\dots,n\}$ and zero otherwise. This allows us to write
    \[
    \mathcal K_\mathcal P(j)=\sum_{k=1}^n|\<U^*e_j,\iota_n\mathbf{e_k}\>|^2=\sum_{k=1}^n|\<\iota_n^*U^*e_j,\mathbf{e_k}\>|^2=\sum_{k=1}^n|\<v_j,\mathbf{e_k}\>|^2=\|v_j\|^2,
    \]
    hence
    \[
    \kappa_w(\mathcal P)=\sup_{j\in\N}|w(j)\mathcal K_\mathcal P(j)|=\sup_{j\in\N}\frac{\|v_j\|^2}{p_j}=R. \qedhere
    \]
\end{proof}

\subsection{Proofs for the Frames Extension}\label{subsec:proofs_frames}

In the full-frame setting, we lose the invertibility of $\Sigma$. In order to recover both the reconstruction error estimate and the concentration of $K_{n,\Omega}$ from the Riesz scenario, we need the high-probability eventual-range stability of $\randgram$.

\begin{proof}[Proof of Proposition \ref{rank.stability}]
    Start by noticing that
    \[
    \ran(\Sigma)=\ran(\iota_n^*W^*SS^*W\iota_n)\subseteq \ran(\iota_n^*W^*)=\ran(\iota_n^*W^*W\iota_n),
    \]
    where the last equality follows from Lemma \ref{kerA*A}. Moreover, from the fact that $\ker(S^*)=\ran(S)^\perp=\mathcal S^\perp$ in conjunction with $\mathcal W_n\cap\mathcal S^\perp\subseteq\mathcal W\cap\mathcal S^\perp\subseteq\mathcal S\cap\mathcal S^\perp=\{0\}$ we deduce that $S^*_{|\mathcal W_n}$ is injective. It follows that $\dim(S^*(\mathcal W_n))=\dim(\mathcal W_n)$, but $\mathcal W_n=\ran(W\iota_n)\cong \ran(\iota_n^*W^*W\iota_n)$ and $S^*(\mathcal W_n)=S^*(\ran(W\iota_n))=\ran(S^*W\iota_n)=\ran(U\iota_n)\cong \ran(\iota_n^*U^*U\iota_n)=\ran(\Sigma)$, where both isomorphism are due to Lemma \ref{kerA*A}. We have thus proved that $\dim(\ran(\Sigma))=\dim(\ran(\iota_n^*W^*W\iota_n))$, but since $\ran(\Sigma)\subseteq \ran(\iota_n^*W^*W\iota_n)$ we actually conclude $\ran(\Sigma)=\ran(\iota_n^*W^*W\iota_n)$\\
    As for the second part of the statement, one inclusion is once again easy to see and holds for every $m\in\N$ with probability $1$:
    \[
    \ran(\randgram)=\ran(\iota_n^*U^*Q_\Omega U\iota_n)\subseteq \ran(\iota_n^*U^*)=\ran(\Sigma).
    \]
    Now, let $\delta\in(0,1)$ and $m\ge M_\delta$: thanks to Proposition \ref{lsc.rank}, we have
    \[
    \left\{\|\randgram-\Sigma\|<\lambda_0\right\}\subseteq\left\{\dim(\ran(\randgram))\ge\dim(\ran(\Sigma))\right\},
    \]
    allowing us to conclude that
    \begin{align*}
        \P\left(\ran(\randgram)=\ran(\Sigma)\right)&=\P\left(\left\{\ran(\randgram)\subseteq \ran(\Sigma)\right\}\cap\left\{\dim(\ran(\randgram))\ge\dim(\ran(\Sigma))\right\}\right)\\
        &=\P\left(\dim(\ran(\randgram))\ge\dim(\ran(\Sigma))\right)\\
        &\ge\P\left(\|\randgram-\Sigma\|<\lambda_0\right)\ge1-\delta,
    \end{align*}
     where the last inequality is due to Proposition \ref{prob.conv} (recall that this proposition does not rely on $\{w_k\}_{k\in\N}$ being a Riesz basis).
\end{proof}

We are now ready to prove the main result in the case where the sensing system is an arbitrary frame.
\begin{proof}[Proof of Theorem \ref{main+frames}]
    Fix $\delta\in(0,1)$ and define $\tilde f$ as in \eqref{tilde.f.frames}. The proof proceeds, mutatis mutandis, like the proof of Theorem \ref{mainresult}. Decompose $f=P_{\mathcal W_n}f+P_{\mathcal W_n^\perp}f$.
    \begin{align*}
        \tilde f=W\iota_n\randgram^\dagger\iota_n^*U^*Q_\Omega S^*P_{\mathcal W_n}f+W\iota_n\randgram^\dagger\widehat C_\Omega f.
    \end{align*}
    Focus on the first summand $\tilde f_1:=W\iota_n\randgram^\dagger\iota_n^*U^*Q_\Omega S^*P_{\mathcal W_n}f$. Now, $W\iota_n$ is not necessarily injective here, but we still have the property $W\iota_n(\iota_n^*W^*W\iota_n)^\dagger\iota_n^*W^*=P_{\mathcal W_n}$:
    \[
    f_1=W\iota_n\randgram^\dagger\iota_n^*U^*Q_\Omega S^*W\iota_n(\iota_n^*W^*W\iota_n)^\dagger\iota_n^*W^*f=W\iota_n\randgram^\dagger\randgram(\iota_n^*W^*W\iota_n)^\dagger\iota_n^*W^*f.
    \]
    At this point we cannot cancel out $\randgram^\dagger$ and $\randgram$ as we did in the frame-to-Riesz case: instead, we must use $\randgram^\dagger\randgram=P_{\ran(\randgram^*)}=P_{\ran(\randgram)}$ to obtain
    \[
    \tilde f_1=W\iota_nP_{\ran(\randgram)}(\iota_n^*W^*W\iota_n)^\dagger\iota_n^*W^*f.
    \]
    We know from Proposition \ref{rank.stability} that for every $m\ge M_\delta$ we have $P_{\ran(\randgram)}=P_{\ran(\iota_n^*W^*W\iota_n)}$ with probability at least $1-\delta$, and every logical implication that will follow from this will be true on a set of at least this probability. In particular,
    \[
    \tilde f_1=W\iota_n(\iota_n^*W^*W\iota_n)^\dagger\iota_n^*W^*f=P_{\mathcal W_n}f,
    \]
    hence
    \[
    \tilde f=P_{\mathcal W_n}f+W\iota_n\randgram^\dagger\widehat C_\Omega f.
    \]
    It follows that
    \[
    f-\tilde f=P_{\mathcal W_n^\perp}f-W\iota_n\randgram^\dagger\widehat C_\Omega f.
    \]
    To conclude, by the Pythagorean theorem,
    \begin{align*}
        \|f-\tilde f\|&=(\|P_{\mathcal W_n^\perp}f\|^2+\|W\iota_n\randgram^\dagger\widehat C_\Omega f\|^2)^{\frac{1}{2}}\\
        &\le\|P_{\mathcal W_n^\perp}f\|(1+K_{n,\Omega}^2)^{\frac{1}{2}}.\qedhere
    \end{align*}
\end{proof}
We follow up with the concentration inequality for the factor $K_{n,\Omega}$ that is analogous to the one we deduced in the frame-to-Riesz case. The proof is very similar, although one must be careful with the $\varepsilon$-dependencies, the event inclusions and the crucial role played by the eventual range stability.
\begin{proof}[Proof of Theorem \ref{K_n,omega+frames}]
    The proof proceeds almost exactly like the one for Theorem \ref{K_n,omega}, except that every instance of operator inversion is replaced with pseudo-inversion. We stress that instead of using Lemma \ref{cont.inv} to infer
    \[
    E_1\supseteq\left\{\|\randgram-\Sigma\|<\frac{\varepsilon}{4\|\Sigma^{-1}\|^2\Lambda\sqrt D}\right\},
    \]
    we resort to Lemma \ref{cont.pseudo-inv}, which states that, for the modified event $E_1^\dagger=\left\{\|\randgram^\dagger-\Sigma^\dagger\|<\frac{\varepsilon}{\Lambda\sqrt D}\right\}$, we have
    \begin{align}\label{fake.intersect}
        E_1^\dagger\supseteq\left\{\|\randgram-\Sigma\|<\frac{\varepsilon}{4\|\Sigma^\dagger\|^2\Lambda\sqrt D}\right\}\cap\left\{\ran(\randgram)=\ran(\Sigma)\right\}.
    \end{align}
However, we saw in the proof of Proposition \ref{rank.stability} that 
    \[
    \left\{\|\randgram-\Sigma\|<\lambda_0\right\}\subseteq\left\{\ran(\randgram)=\ran(\Sigma)\right\},
    \]
    and since $\frac{\varepsilon}{4\|\Sigma^\dagger\|^2\Lambda\sqrt D}\le\lambda_0=\frac{1}{\|\Sigma^\dagger\|}$ (indeed, $\varepsilon<2\|\Sigma^\dagger\|\Lambda\sqrt D$ by assumption), the intersection in \eqref{fake.intersect} only amounts to $\left\{\|\randgram-\Sigma\|<\frac{\varepsilon}{4\|\Sigma^\dagger\|^2\Lambda\sqrt D}\right\}$, and the proof can now be finished exactly like the proof of Theorem \ref{K_n,omega}.
\end{proof}

\subsection{Proof of the Stable Recovery of Analytic Functions}
We now focus on the proof of the recovery of analytic functions from Fourier measurements. 
\begin{proof}[Proof of Theorem \ref{analytic}]
    Since both the Fourier basis and the Legendre polynomial basis are orthonormal, both the sampling rate $m\ge\frac{8}{3}n\log(2n/\delta)$ and the existence and uniqueness of the solution to the least-squares problem follow from Corollary \ref{cor:onb_nlogn} if the leverage-score distribution is implemented. We now wish to evaluate the square norm of $v_l=\iota_n^*U^*e_l=\iota_n^*W^*s_l$ for each $l\in\N$ in this Fourier-Legendre setting. Since
    \[
    W^*s_l=(\<s_l,w_k\>)_{k\in\N},
    \]
    we have that $v_l=(\<s_l,w_k\>)_{k=1}^n\in\C^n$, hence
    \[
    \|v_l\|^2=\sum_{k=1}^n|\<s_l,w_k\>|^2=\sum_{k=1}^n|\widehat w_k(\sigma(l))|^2=\sum_{k=0}^{n-1}|\widehat w_{k+1}(\sigma(l))|^2.
    \]
    The Fourier coefficients of the Legendre polynomials are known to have a closed form in terms of the spherical Bessel function of first kind \cite[Formula 10.1.14]{abramowitz1948handbook}:
    \[
    \widehat w_{k+1}(\sigma(l))=i^k\sqrt{2k+1}\mathcal\,\,j_k(-\pi\sigma(l)).
    \]
    Ultimately, we obtain
    \[
    p_l=\frac{\|v_l\|^2}{n}=\frac{1}{n}\sum_{k=0}^{n-1}(2k+1)\left|j_k(-\sigma(l)\pi)\right|^2.
    \]
    We also deduce the error estimate $\|f-\tilde f\|\lesssim\|P_{\mathcal W_n^\perp}f\|(1+K_{n,\Omega}^2)^{\frac{1}{2}}$. To conclude the proof, the only thing left to prove is that there exists $\rho>1$ such that $\|P_{\mathcal W_n^\perp}f\|_{L^2([-1,1])}\lesssim\sqrt n\rho^{-n}$. Since the normalized Legendre polynomials form an orthonormal basis of $L^2([-1,1])$, we have
    \[
    \|P_{\mathcal W_n^\perp}f\|_{L^2([-1,1])}^2=\sum_{k=n+1}^{+\infty}|\<f,w_k\>_{L^2([-1,1])}|^2.
    \]
    Let $P_k:=(k+\frac{1}{2})^{-\frac{1}{2}}w_{k+1}$ be the (unnormalized) $k$-th Legendre polynomial. Since $f$ is analytic on $[-1,1]$, there exists $\rho>1$ such that $f$ is holomorphic on the interior of the ellipse
    \[
    \mathcal E_\rho:=\left\{\frac{z+z^{-1}}{2}\st|z|=\rho\right\},
    \]
    that is, the ellipse on the complex plane centered at zero whose axes have length $\rho+\rho^{-1}$ and $\rho-\rho^{-1}$ and whose foci are located at $\{-1,1\}$ (such curves are called Bernstein ellipses). Then it's known \cite{davis1975interpolation} that $|\<f,P_k\>_{L^2([-1,1])}|\le c\rho^{-k}$ for some constant $c>0$ that depends upon $\rho$ but not on $k$. We thus get
    \begin{align*}
        \sum_{k=n+1}^{+\infty}|\<f,w_k\>_{L^2([-1,1])}|^2&=\sum_{k=n}^{+\infty}\left(k+\frac{1}{2}\right)|\<f,P_k\>_{L^2([-1,1])}|^2\\
        &\le c^2\sum_{k=n}^{+\infty}\left(k+\frac{1}{2}\right)\rho^{-2k}\\
        &\le 2c^2\sum_{k=n}^{+\infty}k\rho^{-2k}\\
        &=2c^2\rho^{-2n}\frac{n-n\rho^{-2}+\rho^{-2}}{(1-\rho^{-2})^2}\\
        &\le\frac{2c^2}{(1-\rho^{-2})^2}n\rho^{-2n}.
    \end{align*}
    Ultimately, we deduce that 
    \[
    \|P_{\mathcal W_n^\perp}f\|_{L^2([-1,1])}\le C\sqrt n\rho^{-n}. \qedhere
    \]
\end{proof}

\section{Conclusions}\label{sec:conclusions}

This work has introduced a fully stochastic approach to generalized sampling. We have demonstrated that, by employing a suitable randomized sampling strategy, stable recovery is guaranteed with high probability at a  near linear sample complexity, even when operating with highly redundant sensing systems such as frames. Establishing these guarantees required us to derive sharp concentration estimates for both the empirical Gram matrix and the empirical cross-term operator, culminating in a new variant of Bernstein's inequality for operators acting between potentially distinct, infinite-dimensional Hilbert spaces. Furthermore, we provided a comprehensive comparative analysis between leverage-score sampling and Christoffel sampling within this abstract frame setting. Finally, we showcased the practical viability of this machinery through the Fourier-Legendre approximation problem, where the stochastic framework natively overcomes deterministic ill-conditioning to achieve near-exponential convergence rates.

While this framework significantly advances the stability and efficiency of infinite-dimensional inverse problems, it is important to delineate its domain of applicability. First, the universality of the $m \gtrsim n \log n$ sampling rate relies on the geometric inclusion $\mathcal W\subseteq\mathcal S$. Second, the logarithmic oversampling factor is an intrinsic feature of independent random sampling. Achieving a strictly linear rate would require departing from i.i.d. extractions—for instance, by employing greedy selection methods \cite{dolbeault2024randomized} or determinantal point processes \cite{determinantal}. However, the former incurs higher computational costs, while the latter generally requires a reproducing kernel Hilbert space structure, a restriction our abstract operator framework was deliberately designed to bypass. We also emphasize that leverage-score sampling is a cure for basis mismatch and geometric redundancy, not for intrinsic ill-posedness; if the underlying physical problem is severely ill-posed, the condition number of the Gram matrix will naturally reflect this sensitivity. Finally, the current theory is developed for linear inverse problems governed by well-defined frame bounds.

These considerations naturally suggest new directions for future research. A primary open theoretical question is whether the log-linear sampling rate can be universally achieved without the subspace condition $\mathcal W\subseteq\mathcal S$. Furthermore, extending this stochastic GS framework to non-linear inverse problems represents a challenging but highly relevant direction. Within the context of linear problems, another promising path involves generalizing the nature of the measurements themselves. Instead of restricting the analysis to continuous linear functionals defined on the entire Hilbert space $\mathcal H$, one could adapt the framework to handle ``indirect data'' \cite{adcock2020frames}—that is, continuous linear functionals $\ell\colon G\rightarrow\C$ defined only on a proper, dense subspace $G\subseteq\mathcal H$. Because such measurements cannot be globally represented by standard inner products, this extension would encompass crucial real-world scenarios, such as the pointwise evaluation of continuous functions in $L^2$ when a reproducing kernel structure is unavailable.

\section*{Acknowledgements}

This material is based upon work supported by the Air Force Office of Scientific Research under award number FA8655-23-1-7083. The research was supported in part by the MIUR Excellence Department Project awarded to Dipartimento di Matematica, Università di Genova, CUP D33C23001110001.

\bibliography{biblio}
\bibliographystyle{plain}

\appendix

\section{Technical Lemmas}
\label{sec:appendix}

In this section, we collect several auxiliary results from functional analysis and operator theory. Although their underlying principles are well known, we state and prove them here in the precise formulations required for our subsequent analysis.

\begin{lemma}\label{kerA*A}
    Let $\mathcal H,\mathcal K$ be Hilbert spaces and let $A\colon\mathcal H\longrightarrow\mathcal K$ be a linear operator. Then $\ker(A^*A)=\ker(A)$ and $\overline{\ran(A^*)}=\overline{\ran(A^*A)}$. In particular, if $A$ is bounded and has closed range, then $\ran(A^*)=\ran(A^*A)\cong \ran(A)$.
\end{lemma}
\begin{proof}
    Clearly, if $Ax=0$, then $A^*Ax=0$: this proves that $\ker(A)\subseteq \ker(A^*A)$. Conversely, assume that $x\in \ker(A^*A)$. Then
    \[
    0=\<A^*Ax,x\>=\<Ax,Ax\>=\|Ax\|^2,
    \]
    which means that $Ax=0$ and proves $\ker(A)=\ker(A^*A)$. From this we immediately deduce
    \[
    \overline{\ran(A^*A)}=(\ran(A^*A)^\perp)^\perp=\ker(A^*A)^\perp=\ker(A)^\perp=\overline{\ran(A^*)}.
    \]
    If we further assume that $A$ is bounded and has closed range, then the first isomorphism theorem grants
    \begin{equation*}
    \ran(A^*)=\ran(A^*A)\cong\mathcal H\diagup \ker(A^*A)=\mathcal H\diagup \ker(A)\cong \ran(A). \qedhere
    \end{equation*}
\end{proof}

\begin{lemma}\label{below.bound.adjoint}
    Let $T\colon\mathcal H\longrightarrow\mathcal K$ be a bounded, injective operator with closed range between Hilbert spaces and let $c>0$ such that $\|Th\|\ge c\|h\|$ for every $h\in\mathcal H$. Then $\|T^*Th\|\ge c\|Th\|$ for every $h\in\mathcal H$.
\end{lemma}
\begin{proof}
    Let $h\in\mathcal H$. We can assume $h\neq0$, otherwise the claim is trivial. Then
    \[
    \|T^*Th\|\|h\|\ge|\<T^*Th,h\>|=\|Th\|^2\ge c\|h\|\|Th\|.
    \]
    The claim follows by dividing both sides by $\|h\|$.
\end{proof}

\begin{lemma}\label{norm.of.inverse}
    Let $X$ be a Banach space. Then the set of bounded bijective linear operators is open in $\mathcal B(X)$. More precisely, let $A,B\in\mathcal B(X)$ and suppose that $A$ is bijective. If $\|B-A\|<\|A^{-1}\|^{-1}$, then $B$ is bijective and $\|B^{-1}\|\le\frac{1}{\|A^{-1}\|^{-1}-\|B-A\|}$.
\end{lemma}
\begin{proof}
    One has:
    \[
    \|I-A^{-1}B\|=\|A^{-1}(A-B)\|\le\|A^{-1}\|\|B-A\|<1.
    \]
    By virtue of the Neumann series criterion, $A^{-1}B$ is bijective and
    \[
    (A^{-1}B)^{-1}=\sum_{k=0}^{+\infty}(I-A^{-1}B)^k.
    \]
    It follows that $B=A(A^{-1}B)$ is also bijective and
    \begin{align*}
        \|B^{-1}\|=\|(A^{-1}B)^{-1}A^{-1}\|&\le\|(A^{-1}B)^{-1}\|\|A^{-1}\|\\
        &\le\|A^{-1}\|\sum_{k=0}^{+\infty}\|I-A^{-1}B\|^k\\
        &\le\|A^{-1}\|\sum_{k=0}^{+\infty}(\|A^{-1}\|\|B-A\|)^k\\
        &=\frac{\|A^{-1}\|}{1-\|A^{-1}\|\|B-A\|}\\
        &=\frac{1}{\|A^{-1}\|^{-1}-\|B-A\|}.\qedhere
    \end{align*}
\end{proof}

\begin{prop}\label{lsc.rank}
    Let $A\in\C^{n\times n}$ be self-adjoint and let $r:=\operatorname{rk}(A)$. Let $\lambda_r$ be the eigenvalue of $A$ with the smallest non-zero absolute value. Then for every $B\in\C^{n\times n}$ such that $\|B-A\|<|\lambda_r|$ one has $\operatorname{rk}(B)\ge r$.
\end{prop}
\begin{proof}
    Let $\sigma(A)=\{\lambda_1,\dots,\lambda_n\}\subseteq\R$ such that $|\lambda_1|\ge\dots\ge|\lambda_r|>\lambda_{r+1}=\dots=\lambda_n=0$. We can find an orthonormal system $\{x_1,...,x_r\}\subseteq\C^n$ of eigenvectors for $A$, where $Ax_j=\lambda_j x_j$ for every $j\in\{1,\dots,r\}$. Let $a_1,\dots,a_r\in\C$ such that $\sum_{j=1}^r a_jBx_j=0$. On one hand, we have:
    \begin{align*}
        \left\|\sum_{j=1}^r a_j Ax_j\right\|^2&=\left\|\sum_{j=1}^r a_j (A-B)x_j\right\|^2=\left\|(A-B)\left(\sum_{j=1}^r a_jx_j\right)\right\|^2\\
        &\le\|B-A\|^2\left\|\sum_{j=1}^r a_jx_j\right\|^2=\|B-A\|^2\sum_{j=1}^r|a_j|^2.
    \end{align*}
    If we were to assume $\sum_{j=1}^r|a_j|^2\neq0$, we would have
    \[
    \left\|\sum_{j=1}^r a_j Ax_j\right\|^2<|\lambda_r|^2\sum_{j=1}^r|a_j|^2,
    \]
    but also
    \begin{align*}
        \left\|\sum_{j=1}^r a_j Ax_j\right\|^2=\left\|\sum_{j=1}^r a_j\lambda_jx_j\right\|^2=\sum_{j=1}^r |a_j\lambda_j|^2\ge|\lambda_r|^2\sum_{j=1}^r|a_j|^2,
    \end{align*}
    a contradiction. It follows that $a_1,\dots,a_r=0$, hence $\{Bx_1,...,Bx_r\}$ is linearly independent, so that $\operatorname{rk}(B)\ge r$.
\end{proof}

\begin{lemma}[Continuity of operator composition]\label{cont.comp}
    Let $X,Y,Z$ be Banach spaces and let $B\in\mathcal B(X,Y),A\in\mathcal B(Y,Z)$. Then for every $\varepsilon>0$ and every $D\in\mathcal B(X,Y),C\in\mathcal B(Y,Z)$ satisfying
    \[
    \|C-A\|,\|D-B\|<\delta:=\min\left\{1,\frac{\varepsilon}{1+\|A\|+\|B\|}\right\}
    \]
    one has $\|CD-AB\|<\varepsilon.$
\end{lemma}
\begin{proof}
    Write
    \begin{align*}
        CD-AB&=(C-A)(D-B)+CB+AD-2AB\\
        &=(C-A)(D-B)+(C-A)B+A(D-B).
    \end{align*}
    This immediately implies
    \[
    \|CD-AB\|<\delta^2+\delta\|A\|+\delta\|B\|=\delta(\delta+\|A\|+\|B\|)\le\delta(1+\|A\|+\|B\|)\le\varepsilon,
    \]
    where we used the fact that $\delta\le1$.
\end{proof}

\begin{lemma}[Continuity of operator inversion]\label{cont.inv}
    Let $X$ be a Banach space and let $A$ be a bounded, invertible operator on $X$. Then for every $\varepsilon>0$ and every $B\in\mathcal B(X)$ satisfying
    \[
    \|B-A\|<\delta:=\min\left\{\frac{1}{2\|A^{-1}\|},\frac{\varepsilon}{4\|A^{-1}\|^2}\right\}
    \]
    one has that $B$ is invertible and $\|B^{-1}-A^{-1}\|<\varepsilon$.
\end{lemma}
\begin{proof}
    Since $\|B-A\|<\frac{1}{2}\|A^{-1}\|^{-1}<\|A^{-1}\|^{-1}$, we already know from Lemma \ref{norm.of.inverse} that $B$ is invertible and
    \[
    \|B^{-1}\|\le\frac{1}{\|A^{-1}\|^{-1}-\|B-A\|}<\frac{1}{\|A^{-1}\|^{-1}-\frac{1}{2}\|A^{-1}\|^{-1}}=2\|A^{-1}\|.
    \]
    But then
    \begin{align*}
        \|B^{-1}-A^{-1}\|&=\|B^{-1}(A-B)A^{-1}\|\\
        &\le\|B^{-1}\|\|B-A\|\|A^{-1}\|\\
        &<2\delta\|A^{-1}\|^2\le\frac{\varepsilon}{2}<\varepsilon.\qedhere
    \end{align*}
\end{proof}
\begin{remark}
    It's easy to see from the last step of the previous proof that taking 
    \[
    \delta=\min\left\{\frac{1}{2\|A^{-1}\|},\frac{\varepsilon}{2\|A^{-1}\|^2}\right\}
    \]
    would be a more natural and optimal choice. However, we prefer to use $\frac{\varepsilon}{4\|A^{-1}\|^2}$ instead of $\frac{\varepsilon}{2\|A^{-1}\|^2}$ for consistency with the result concerning the continuity of the pseudo-inversion (see Lemma \ref{cont.pseudo-inv}).
\end{remark}

\begin{lemma}\label{diff.proj.norm}
    Let $\mathcal H$ be a Hilbert space and let $P,Q$ be two orthogonal projection of $\mathcal H$. Then
    \[
    \|P-Q\|=\max\{\|PQ^\perp\|,\|QP^\perp\|\}.
    \]
\end{lemma}
\begin{proof}
    On one hand, we have
    \[
    PQ^\perp=P-PQ=P(P-Q),
    \]
    which immediately entails $\|PQ^\perp\|\le\|P-Q\|$ and, with similar reasoning, $\|QP^\perp\|\le\|P-Q\|$ as well. On the other hand, 
    \[
    P-Q=P-PQ+PQ-Q=PQ^\perp-P^\perp Q,
    \]
    and for every $x\in\mathcal H$ one has
    \begin{align*}
    \|(P-Q)x\|^2&=\|PQ^\perp x-P^\perp Qx\|^2=\|PQ^\perp x\|^2+\|P^\perp Qx\|^2\\
    &\le\|PQ^\perp\|^2\|Q^\perp x\|^2+\|P^\perp Q\|^2\|Qx\|^2\\
    &\le\left(\max\{\|PQ^\perp\|,\|P^\perp Q\|\}\right)^2(\|Q^\perp x\|^2+\|Qx\|^2)\\
    &=\left(\max\{\|PQ^\perp\|,\|P^\perp Q\|\}\right)^2\|x\|^2,
    \end{align*}
    and this proves $\|P-Q\|\le\max\{\|PQ^\perp\|,\|P^\perp Q\|\}$. Finally, $\|P^\perp Q\|=\|(P^\perp Q)^*\|=\|QP^\perp\|$.
\end{proof}

\begin{lemma}[Continuity of pseudo-inversion]\label{cont.pseudo-inv}
    Let $\mathcal H$ be a Hilbert space and fix a closed subspace $Y\subseteq\mathcal H$. Let $A\in\mathcal B(\mathcal H)$ such that $\ran(A)=Y$. Then for every $\varepsilon>0$ and every $B\in\mathcal B(\mathcal H)$ such that $\ran(B)=Y$ satisfying
    \[
    \|B-A\|<\delta:=\min\left\{\frac{1}{2\|A^\dagger\|},\frac{\varepsilon}{4\|A^\dagger\|^2}\right\}
    \]
    one has $\|B^\dagger-A^\dagger\|<\varepsilon$.
\end{lemma}
\begin{proof}
    It's known that $\ker(B^\dagger)=\ker(B^*)$: this allows us to write
    \[
    B^\dagger=B^\dagger P_{\ker(B^*)^\perp}=B^\dagger P_{Y}=B^\dagger P_{Y}-P_{\ran(B^*)}A^\dagger+P_{\ran(B^*)}A^\dagger.
    \]
    Use the properties $P_{Y}=AA^\dagger$ and $P_{\ran(B^*)}=B^\dagger B$ to write
    \[
    B^\dagger=B^\dagger AA^\dagger-B^\dagger BA^\dagger+P_{\ran(B^*)}A^\dagger=B^\dagger(A-B)A^\dagger+P_{\ran(B^*)}A^\dagger.
    \]
    It follows that $\|B^\dagger\|\le\|B^\dagger\|\|A^\dagger\|\delta+\|A^\dagger\|\le\frac{1}{2}\|B^\dagger\|+\|A^\dagger\|$, hence $\|B^\dagger\|\le2\|A^\dagger\|$. From the previous computations we also deduce
    \begin{align*}
        B^\dagger-A^\dagger&=B^\dagger(A-B)A^\dagger+P_{\ran(B^*)}A^\dagger-A^\dagger\\
        &=B^\dagger(A-B)A^\dagger+P_{\ran(B^*)}A^\dagger-P_{\ran(A^*)}A^\dagger\\
        &=(B^\dagger(A-B)+P_{\ran(B^*)}-P_{\ran(A^*)})A^\dagger,
    \end{align*}
    where we used that $\ran(A^\dagger)=\ran(A^*)$. By Lemma \ref{diff.proj.norm}, we have 
    \begin{align*}
        \|P_{\ran(B^*)}-P_{\ran(A^*)}\|&=\max\{\|P_{\ran(B^*)}P_{\ran(A^*)^\perp}\|,\|P_{\ran(A^*)}P_{\ran(B^*)^\perp}\|\}\\
        &=\max\{\|B^\dagger BP_{\ker(A)}\|,\|A^\dagger AP_{\ker(B)}\|\}\\
        &=\max\{\|B^\dagger(B-A)P_{\ker(A)}\|,\|A^\dagger (B-A)P_{\ker(B)}\|\}\\
        &\le\max\{\|A^\dagger\|,\|B^\dagger\|\}\|B-A\|<2\delta\|A^\dagger\|.
    \end{align*}
    This allows us to conclude:
    \begin{align*}
        \|B^\dagger-A^\dagger\|&<(\|B^\dagger\|\delta+\|P_{\ran(B^*)}-P_{\ran(A^*)}\|)\|A^\dagger\|\\
        &\le(2\delta\|A^\dagger\|+2\delta\|A^\dagger\|)\|A^\dagger\|\\
        &=4\delta\|A^\dagger\|^2\le\varepsilon.\qedhere
    \end{align*}
\end{proof}

\begin{lemma}\label{cont.pseudo}
    Let $T\colon\mathcal H\longrightarrow\mathcal K$ be a linear operator between Hilbert spaces with closed range such that $\|Tx\|\ge c\|x\|$ for some $c>0$ for every $x\in\ker(T)^\perp$. Then the pseudo-inverse $T^\dagger\colon\mathcal K\longrightarrow\mathcal H$ is bounded and $\|T^\dagger\|\le\frac{1}{c}$.
\end{lemma}
\begin{proof}
    Let $y\in\mathcal K$. By a well-known property of pseudo-inverses, $P_{R(T)}y=TT^\dagger y$. Taking the norm of both sides, we get
    \[
    \|y\|\ge\|P_{R(T)}y\|=\|TT^\dagger y\|\ge c\|T^\dagger y\|,
    \]
    where we used the fact that $\|P_{R(T)}y\|\le\|y\|$ and $T^\dagger y\in\ran(T^\dagger)=\ker(T)^\perp$. The claim follows by dividing both sides by $c$.
\end{proof}

We conclude with basic properties of the Hermitian dilation of an operator (see Definition~\ref{def:Herm}).

\begin{prop}\label{hermitian.dilation}
    Let $T\colon\mathcal H_1\longrightarrow\mathcal H_2$ be a bounded operator between Hilbert spaces. Then:
    \begin{enumerate}
        \item $H(T)\in\mathcal B(\mathcal H_1\oplus\mathcal H_2)$ and $\|H(T)\|_{\mathcal B(\mathcal H_1\oplus\mathcal H_2)}=\|T\|_{\mathcal B(\mathcal H_1,\mathcal H_2)}$.
        \item $H(T)$ is self-adjoint.
        \item $H\colon\mathcal B(\mathcal H_1,\mathcal H_2)\longrightarrow\mathcal B(\mathcal H_1\oplus\mathcal H_2)$ is $\R$-linear.
    \end{enumerate}
\end{prop}
\begin{proof}
    \begin{enumerate}
        \item $H(T)$ is clearly linear. Moreover, for every $(x,y)\in\mathcal H_1\oplus\mathcal H_2$ we have
        \begin{align*}
            \|H(T)(x,y)\|^2_{\mathcal H_1\oplus\mathcal H_2}&=\|(T^*y,Tx)\|^2_{\mathcal H_1\oplus\mathcal H_2}\\
            &=\|T^*y\|^2_{\mathcal H_1}+\|Tx\|^2_{\mathcal H_2}\\
            &\le\|T\|_{B(\mathcal H_1,\mathcal H_2)}^2(\|x\|_{\mathcal H_1}^2+\|y\|_{\mathcal H_2}^2)\\
            &=\|T\|_{B(\mathcal H_1,\mathcal H_2)}^2\|(x,y)\|_{\mathcal H_1\oplus\mathcal H_2}^2.
        \end{align*}
        It follows that $H(T)$ is bounded and $\|H(T)\|_{\mathcal B(\mathcal H_1\oplus\mathcal H_2)}\le\|T\|_{\mathcal B(X,Y)}$. Conversely, for every $x\in\mathcal H_1$ one has
        \begin{align*}
            \|Tx\|_{\mathcal H_2}^2&=\|T^*0\|_{\mathcal H_1}^2+\|Tx\|_{\mathcal H_2}^2\\
            &=\|(T^*0,Tx)\|_{\mathcal H_1\oplus\mathcal H_2}^2\\
            &=\|H(T)(x,0)\|_{\mathcal H_1\oplus\mathcal H_2}^2\\
            &\le\|H(T)\|_{\mathcal B(\mathcal H_1\oplus\mathcal H_2)}^2\|(x,0)\|_{\mathcal H_1\oplus\mathcal H_2}^2\\
            &=\|H(T)\|_{\mathcal B(\mathcal H_1\oplus\mathcal H_2)}^2\|x\|_{\mathcal H_1}^2,
        \end{align*}
        and this concludes that $\|H(T)\|_{\mathcal B(\mathcal H_1\oplus\mathcal H_2)}=\|T\|_{\mathcal B(\mathcal H_1,\mathcal H_2)}$.
        \item For every $(x_1,y_1),(x_2,y_2)\in\mathcal H_1\oplus\mathcal H_2$ one has
        \begin{align*}
            \<(x_1,y_1),H(T)(x_2,y_2)\>_{\mathcal H_1\oplus\mathcal H_2}&=\<(x_1,y_1),(T^*y_2,Tx_2)\>_{\mathcal H_1\oplus\mathcal H_2}\\
            &=\<x_1,T^*y_2\>_{\mathcal H_1}+\<y_1,Tx_2\>_{\mathcal H_2}\\
            &=\<Tx_1,y_2\>_{\mathcal H_2}+\<T^*y_1,x_2\>_{\mathcal H_1}\\
            &=\<(T^*y_1,Tx_1),(x_2,y_2)\>_{\mathcal H_1\oplus\mathcal H_2}\\
            &=\<H(T)(x_1,y_1),(x_2,y_2)\>_{\mathcal H_1\oplus\mathcal H_2},
        \end{align*}
        hence $H(T)^*=H(T)$.
        \item Let $T_1,T_2\in\mathcal B(\mathcal H_1,\mathcal H_2)$ and let $\alpha,\beta\in\R$. For every $(x,y)\in\mathcal H_1\oplus\mathcal H_2$ we have
        \begin{align*}
            H(\alpha T_1+\beta T_2)(x,y)&=((\alpha T_1+\beta T_2)^*y,(\alpha T_1+\beta T_2)x)\\
            &=(\alpha T_1^*y+\beta T_2^*y,\alpha T_1x+\beta T_2x)\\
            &=\alpha(T_1^*y,T_1x)+\beta(T_2^*y;T_2x)\\
            &=\alpha H(T_1)(x,y)+\beta H(T_2)(x,y).
        \end{align*}
        This proves $H(\alpha T_1+\beta T_2)=\alpha H(T_1)+\beta H(T_2)$ (note that $H$ is not $\C$-linear, as in the second equality we would have had to consider the complex conjugate of $\alpha$ and $\beta$).\qedhere
    \end{enumerate}
\end{proof}

\end{document}